\documentclass[10pt]{amsart}

\usepackage{amsmath, amsthm, amssymb, amsfonts, verbatim}

\usepackage[active]{srcltx}
\usepackage{array}
\usepackage[bookmarks]{hyperref}
\usepackage{epsfig}
\usepackage{color}
\usepackage{comment}
\usepackage{setspace}
\usepackage{textcomp}
\usepackage{url}
\usepackage{multirow}
\usepackage{stmaryrd}
\usepackage{pdfsync}
\usepackage{breqn}

\numberwithin{equation}{section}

\newcommand{\pmat}[1]{\begin{pmatrix} #1 \end{pmatrix}}
\newcommand{\case}[1]{\begin{cases} #1 \end{cases}}

\newcommand{\beq} {\begin{equation}}
\newcommand{\eeq} {\end{equation}}
\newcommand{\bdm} {\begin{displaymath}}
\newcommand{\edm} {\end{displaymath}}

\newcommand{\bit}{\begin{itemize}}
\newcommand{\eit}{\end{itemize}}
\newcommand{\bde}{\begin{description}}
\newcommand{\ede}{\end{description}}

\newcommand{\ben}{\begin{enumerate}}
\newcommand{\een}{\end{enumerate}}
\newcommand{\algn}[1]{\begin{align} #1 \end{align}}
\newcommand{\algns}[1]{\begin{align*} #1 \end{align*}}

\newcommand{\barr}{\begin{array}}
\newcommand{\earr}{\end{array}}

\newcommand{\half} {\ensuremath{\frac{1}{2}}}
\newcommand{\mc}[1]{\mathcal{#1}}
\newcommand{\LRp}[1]{\left( #1 \right)}

\newcommand{\LRa}[1]{\left< #1 \right>}
\newcommand{\mapover}[1]{\stackrel{#1}{\longrightarrow}}

\newcommand{\jump}[1] {\ensuremath{\left[\![{#1}]\!\right]}}

\newcommand{\ra}{\rightarrow}

\newcommand{\R}{{\mathbb R}}

\newcommand{\grad}{\operatorname{grad}}

\renewcommand{\div}{\operatorname{div}}

\newcommand{\tr}{\operatorname{tr}}

\newcommand{\supp}{\operatorname{supp}}

\newcommand{\bs}{\boldsymbol}

\newcommand{\lap}{\Delta}
\newcommand{\pd}{\partial}

\newtheorem{theorem}{Theorem}[section]
 
\newtheorem{lemma}[theorem]{Lemma}

\newcommand{\mfz}{\mathfrak{Z}}
\newcommand{\mfh}{\mathfrak{H}}
\newcommand{\Lm}{\Lambda}
\newcommand{\diam}{\operatorname{diam}}
\newcommand{\Alt}{\operatorname{Alt}}
\newcommand{\vol}{\operatorname{vol}}

\newcommand{\s}{\sigma}

\newcommand{\Q}{\mc{Q}}

\renewcommand{\S}{\mc{S}}
\newcommand{\tnorm}[1]{{\| #1 \|}_{\mc{X}_h}}

\begin{document}

\title[Local coderivatives]{Local coderivatives and approximation of Hodge Laplace problems}\thanks{The research leading to these results has received funding from the  European Research Council under the European Union's Seventh Framework Programme (FP7/2007-2013) / ERC grant agreement 339643. }

\author{Jeonghun J. Lee} 
\address{The Institute for Computational Engineering and Sciences, University of Texas at Austin, Austin, TX 78712, USA}
\email{jeonghun@ices.utexas.edu}
\urladdr{}
\author{Ragnar Winther}
\address{Department of Mathematics,
University of Oslo, 0316 Oslo, Norway}
\email{rwinther@math.uio.no}
\urladdr{http://www.mn.uio.no/math/personer/vit/rwinther/index.html}
\date{October 26, 2016, revised May 8, 2017}
\maketitle

\begin{abstract}
The standard mixed finite element approximations of Hodge Laplace problems associated with the de Rham complex 
are based on proper  discrete subcomplexes. As a consequence, the exterior derivatives, which are local operators, are computed exactly. 
However, the approximations of the associated coderivatives are nonlocal. 
In fact, this nonlocal property is an inherent consequence of the mixed formulation of these methods, and can 
be argued to be an undesired effect of these schemes. As a consequence, it has been argued, at least in special settings, 
that more local methods may have improved properties. In the present paper, we construct such methods by relying on a careful 
balance between the choice of finite element spaces, degrees of freedom, and numerical integration rules. Furthermore, we establish key convergence 
estimates based on a standard approach of variational crimes.
\end{abstract}

\section{Introduction}
The purpose of this paper is to discuss finite element methods for the Hodge Laplace problems of the de Rham complex where both the approximation 
of the exterior derivative and the associated coderivative are local operators. This is in contrast to the more standard mixed methods for these problems, as described in \cite{{AFW06, AFW10}}, where the coderivative is approximated by a 
nonlocal operator $d_h^*$. To discuss this phenomenon in a more familiar setting, at least for the numerical analysis community, consider the mixed method for the Dirichlet problem
associated to a second order elliptic equation of the form
\begin{equation}\label{Lap}
- \div (\bs{K}\grad  u ) = f \quad \text{in } \Omega, \qquad u|_{\partial \Omega} = 0,
\end{equation}
where the unknown function $u$ is a scalar field defined on a bounded domain  $\Omega$ in  $\R^n$, and $\partial \Omega$ is its boundary.
The coefficient $\bs{K}$ is matrix valued, spatially varying, and uniformly positive definite. 
When $\bs{K}$ is the identity, this problem corresponds to the Hodge Laplace problem studied below in the case when the unknown is an $n$--form.
The standard mixed finite element method for this problem, cf. \cite{BFBook}, takes the form:

Find $(\sigma_h,u_h) \in \Sigma_h \times V_h$ such that 
\algn{
\label{mixed-Lap}
\begin{split}
\LRa{\bs{K}^{-1}\sigma_h, \tau} - \LRa{u_h, \div \tau} &= 0, \qquad \quad \tau \in \Sigma_h,\\
\LRa{\div \sigma_h,v}  &= \LRa{f,v}, \quad v \in V_h,
\end{split}
}
where $\Sigma_h$ and $V_h$ are finite element spaces which are subspaces of
$H(\div, \Omega)$ and $L^2(\Omega)$, respectively, and where $\sigma_h$ is an approximation of
$- \bs{K}\grad u$. Here the notation $\LRa{\cdot, \cdot}$ is used to denote the $L^2$ inner product for both scalar fields 
and vector fields defined on $\Omega$.

For the typical examples we have in mind the finite element space $V_h$ will consist of discontinuous piecewise polynomials 
with respect to a nonoverlapping partition $\mathcal T_h$ of the domain $\Omega$. In this case the finite element method \eqref{mixed-Lap}
is referred to as a locally conservative or volume preserving method, since
\begin{equation}\label{conserve}
\int_{\Omega_0} f \, dx = \int_{\partial \Omega_0} \sigma_h \cdot \nu \, ds
\end{equation}
for any subdomain $\Omega_0$ of $\Omega$ which is a union of elements of $\mathcal T_h$. Here $\nu$ is the outward unit normal to the boundary of $\Omega_0$. In particular, \eqref{conserve} holds if $\Omega_0$ consists of a single  element of $\mathcal T_h$, and reflects 
a local conservation property of the continuous problem. In contrast to this, standard finite element methods for problems of the form 
\eqref{Lap}, based on the Dirichlet principle and subspaces of the Sobolev space $H^1(\Omega)$, will not admit a corresponding local conservation 
property, and this may lead to inaccurate approximations.
For example, for porous medium flow in a strongly heterogeneous and anisotropic setting it has been argued that locally conservative numerical methods give a better local representation of the physics of the problem, and therefore
a qualitatively better approximation, cf. \cite{Aavatsmark-etal-1998a,Aavatsmark-etal-1998b}.
As a consequence, there has been a substantial 
interest in developing conservative schemes. In addition to the mixed method \eqref{mixed-Lap} this 
includes various schemes referred to as 
finite volume schemes \cite{Droniou-2014,Eymard-Gallouet-Herbin-2000}, in particular the multi--point flux approximation schemes \cite{Aavatsmark-2002}, and  mimetic
finite differences \cite{Brezzi-Lipnikov-Shashkov-2005,Brezzi-Lipnikov-Simoncini-2005}.

The mixed method \eqref{mixed-Lap} is a volume preserving discretization in the sense of \eqref{conserve}, and it is based on a 
sound variational principle, the principle of complementary energy. On the other hand, the mixed method \eqref{mixed-Lap} fails to have another local property of the continuous problem since the operator, $u_h \mapsto \sigma_h$, defined by the first equation of \eqref{mixed-Lap}, and which approximates the operator $- \bs{K} \grad$, is nonlocal. 
Since  
$\Sigma_h$ is required to be a subset of  $H(\div, \Omega)$, the inverse of the so--called ``mass matrix", derived from the $L^2$ inner product 
$\LRa{\sigma_h, \tau}$ of the first equation of \eqref{mixed-Lap}, will be nonlocal. In other words, a local perturbation of $u_h$ will in general lead to a global perturbation of $\sigma_h$, and this purely numerical effect is sometimes considered to be undesirable. 
In fact, in many physical applications, the map $u_h \mapsto \sigma_h$ approximates a constitutive law which is represented as a local operator. 
Therefore, a central issue in the construction of many of the alternative finite volume schemes is to obtain volume preserving methods which are also based on local approximations of the fluxes 
$\sigma_h\cdot \nu$, cf. \eqref{conserve}. We should also mention that there is a relation between the desired local properties
described above and so-called mass lumping. This is a procedure which is 
often performed in the setting of time dependent problems, to remove the effect of mass matrices, to obtain explicit 
or simplified time stepping schemes. For examples of such studies we refer to \cite{Cohen-Joly-Roberts-Tordjman-2001,MonkCohen1998} and references given there. However, we will not study time dependent problems in this paper, even if our results can potentially be used in this context.

An early attempt to overcome the locality problem of the mixed method \eqref{mixed-Lap} in the two dimensional case,
and using the lowest order Raviart-Thomas space on a triangular mesh,
 was done in \cite{Baranger-Maitre-Oudin-1996}. 
 The discussion  was restricted to the case $\bs{K} $ equal to the identity.
This approach leads to a so-called two-point flux method.
However, this method has serious defects. 
In particular, in the general setting, where $\bs{K}$ is matrix valued and spatially varying,
the two-point flux method will not always be  consistent,  cf. \cite{Aavatsmark-2002,Aavatsmark2007}.
The multi-point flux approximation schemes were derived
to overcome this problem, and with Darcy flow and reservoir simulation as the main area of application.
We refer to the survey paper \cite{Aavatsmark-2002} by Aavatsmark for more details. The multi-point flux schemes are usually  described in the setting of  finite difference methods. 
However, for the analysis of these finite volume schemes it seems that the most useful approach is 
to be able to relate the schemes properly to a perturbed mixed finite element method, cf. \cite{Bause-Hoffman-Knabner-2010,Droniou-Eymard-2006,Klausen-Winther-2006b,Klausen-Winther-2006a,Wheeler-Yotov-2006}.
An alternative approach to overcome the defects of the two-point flux method was proposed by 
Brezzi et al. \cite{Brezzi-Fortin-Maridi-2006}. They proposed to use the lowest order Brezzi-Douglas-Marini space instead of 
the Raviart-Thomas space, and to perturb that mixed method by introducing a quadrature rule based on vertex values instead of edge values. They also showed satisfactory results in the three dimensional case.
 A similar method was proposed by Wheeler and Yotov \cite{Wheeler-Yotov-2006}, where also quadrilateral 
 grids are studied, and further extensions to hexahedral grids are studied in  \cite{Ingram-Wheeler-Yotov-2010,Wheeler-Xue-Yotov-2012}. 
 
 The results of the present paper can be seen as further generalizations of the results of 
 \cite{Brezzi-Fortin-Maridi-2006,Wheeler-Yotov-2006}. In fact, the mixed method 
\eqref{mixed-Lap} corresponds to a special case  of the finite element methods studied in \cite{AFW06, AFW10} for the 
more general Hodge Laplace problems. 
Furthermore, the lack of locality described above is a common feature of almost  all of these finite element methods.
Therefore, the purpose of the present paper is to construct corresponding perturbations of the 
mixed methods for the Hodge Laplace problems which will overcome the problem of lack of locality, at least in the low order case.
As a consequence, the potential applications of the results of this paper are not restricted to 
Darcy flow and similar problems, but may for example also be used to localize various methods for Maxwell's equations.
We refer to \cite{AFW06, AFW10} for more details on the various  realizations of the Hodge Laplace problems. We should mention that the concept of locality used in this paper refers exclusively  to the discrete coderivative operators. More precisely, this means that our approximations do not rely on 
local, and explicitly defined, discrete Hodge star operators. In this respect, our approach is different from the 
methods of ``discrete exterior calculus'', as presented in \cite{DEC05,Hirani-thesis}, cf. also \cite{Hiptmair-2001}.  By construction, these methods utilize local approximations of the Hodge star operators, and as a consequence both the exterior derivative $d$ and the coderivative $d^*$ are approximated by local operators. However, 
so far a satisfactory convergence theory seems still to be lacking for these methods. In contrast to this, for the methods constructed here we derive convergence results based on a standard approach of finite elements and variational crimes. 

%
%

The present paper is organized as follows. In the next section we will present a brief review
of exterior calculus, the de Rham complex and its discretizations.  In Section 3 we will discuss an abstract error analysis, in the setting of Hilbert complexes, which 
we will find useful in more concrete applications below. Such applications, in the setting of 
finite element discretizations with respect to simplicial meshes, will be discussed in Section 4, while methods based on
cubical meshes are discussed in Section 5. Finally, we summarize our results with some concluding remarks in Section 6.

\section{Preliminaries}\label{prelim}
Throughout this paper we will adopt the language of finite element exterior calculus as in
\cite{{AFW06, AFW10}}.  We assume that $\Omega \subset \R^n$ is a bounded
polyhedral domain, and we will study finite element approximations of differential forms defined on $\Omega$.
More precisely, we consider maps defined on $\Omega$ with values in the space 
$\Alt^k(\R^n)$, the space of alternating $k$--linear maps on $\R^n$. 
For $1 \leq k \leq n$ let $\Sigma(k)$ be the set of increasing injective maps from $\{1, ..., k\}$ to $\{1, ..., n \}$. Then 
we can define an inner product on $\Alt^k(\R^n)$ by the formula
\[
\LRa{ a,b}_{\Alt} = \sum_{\sigma \in \Sigma(k)}
a(e_{\sigma_1},\ldots,e_{\sigma_k})b(e_{\sigma_1},\ldots,e_{\sigma_k}),
\quad a,b \in \Alt^k(\R^n),
\]
where $\sigma_i$ denotes $\sigma(i)$ for $1 \leq i \leq k$ and $\{e_1,\ldots,e_n\}$ is any orthonormal basis of $\R^n$. 
Differential forms are maps defined on a spatial domain $\Omega$ with values in 
$\Alt^k (\R^n)$.
If $u$ is a differential $k$-form and $t_1,  \ldots ,t_k$ are vectors in $\R^n$, then $u_x(t_1, \ldots, t_k)$ denotes the value of 
$u$ applied to the vectors $t_1,  \ldots ,t_k$ at the point $x \in \Omega$. 
The differential form $u$ is an element of the  space $L^2 \Lm^k(\Omega)$ if and only if 
the map
\[
x \mapsto u_x(t_1, \ldots, t_k)
\]
is in $L^2(\Omega)$ for all 
tuples $t_1,  \ldots ,t_k$.
In fact, $L^2 \Lm^k(\Omega)$ is a Hilbert space with inner product given by
\[
\LRa{u,v} = \int_{\Omega} \LRa{u_x,v_x}_{\Alt} \, dx.
\]
The exterior derivative of a $k$-form $u$ is a $(k+1)$-form $du$ given by
\[
du_x(t_1,   \ldots t_{k+1}) = \sum_{j=1}^{k+1} (-1)^{j+1} \partial_{t_j} u_x(t_1, 
\ldots, \hat t_j, \ldots ,t_{k+1}),
 \]
 where $\hat t_j$ implies that $t_j$ is not included, and $\partial_{t_j}$ denote the directional derivative.
The Hilbert space $H\Lambda^k(\Omega)$ is the corresponding  space of 
$k$-forms $u$ on $\Omega$, which is in $L^2 \Lm^k(\Omega)$, and where
its exterior derivative, $du = d^ku$, is also in $L^2 \Lm^{k+1}(\Omega)$.  The $L^2$ version of the de Rham complex then takes
the form
 \[
H\Lambda^0(\Omega)
\xrightarrow{d^0} H\Lambda^1(\Omega) \xrightarrow{d^1}
\cdots \xrightarrow{d^{n-1}}
H\Lambda^n(\Omega).
\]
In the setting of $k$--forms, the  Hodge Laplace problem takes the form 
\begin{equation}\label{hodge-Lap-strong}
Lu = (d^*d + dd^*)u = f,
\end{equation}
where $d = d^k$ is the exterior derivative mapping $k$--forms to $(k+1)$--forms, and the coderivative  $d^*= d_k^*$ can be seen as the formal
adjoint of $d^{k-1}$. Hence, the Hodge Laplace operator $L$ above is more precisely expressed as $L = d_{k+1}^*d^k + d^{k-1}d_k^*$.
A typical model problem studied in \cite{AFW06, AFW10} is of the form \eqref{hodge-Lap-strong} and with appropriate boundary conditions.
The mixed finite element methods are derived from a weak formulation, where $\sigma = d^* u$ is introduced as an auxiliary variable. It is of the form:

Find $(\sigma, u) \in H\Lambda^{k-1}(\Omega) \times H\Lambda^{k}(\Omega)$ such that 
\begin{align}\label{hodge-mixed-1}
\begin{split}
\LRa{\sigma, \tau} - \LRa{u, d ^{k-1}\tau} &= 0, \qquad \quad \tau \in H\Lambda^{k-1}(\Omega), \\
\LRa{d^{k-1}\sigma,v} + \LRa{d^ku,d^kv} &= \LRa{f,v}, \quad v \in H\Lambda^{k}(\Omega).
\end{split}
\end{align}
Here  $\LRa{\cdot,\cdot}$ denotes the inner products of all the spaces of the form  $L^2 \Lm^j(\Omega)$ which appears in the formulation,
i.e., $j = k-1,k,k+1$.
We refer to Sections 2 and 7 of \cite{AFW06} for more details.
We note that  only the exterior derivate $d$ is used explicitly in the weak  formulation above, while the relation $\sigma = d_k^*u$ 
is formulated weakly in the first equation. 
The formulation also contains the proper natural boundary conditions.
The problem \eqref{hodge-mixed-1} with $k=n-1$ corresponds to a weak formulation 
of the elliptic equation \eqref {Lap} in the case when the coefficient $\bs{K}$ is the identity matrix.
In fact, variable coefficients can also easily be included in the weak formulations 
 \eqref{hodge-mixed-1} by changing the $L^2$ inner products, see \cite[Section 7.3]{AFW06}.
 However, throughout most of the discussion  below we will restrict  the discussion 
to the constant coefficient case. But we emphasize that the extension of the discussion to problems with variable coefficients which are piecewise constants with respect to 
the mesh, is indeed straightforward, cf. Section~\ref{conclusion} below.
In general, the solution of the system \eqref{hodge-mixed-1} may not be unique. Depending on the topology of the domain $\Omega$
there may exist nontrivial harmonic forms, i.e., nontrivial elements of the space 
\[
\mfh^k(\Omega) = \{ v \in H\Lm^k (\Omega) \;:\; d v = 0 \text{ and }\LRa{v, d \tau} = 0 \text{ for all }\tau \in H\Lm^{k-1} (\Omega) \} .
\]
Hence, to obtain a system with a unique solution, an extra condition requiring orthogonality with respect to the harmonic forms,
is usually included.

The basic construction in finite element exterior calculus is
of a corresponding subcomplex
\[
V_h^0\xrightarrow{d} V_h^1 \xrightarrow{d}
\cdots \xrightarrow{d}
V_h^n,
\]
where the spaces $V_h^k$ are finite dimensional subspaces of
$H\Lambda^k(\Omega)$. In particular, the discrete spaces should have the property that 
$d (V_h^{k-1}) \subset V_h^k$.
The finite element methods studied in \cite{AFW06, AFW10} are based on the weak formulation \eqref{hodge-mixed-1}.
These methods are obtained by simply replacing the Sobolev spaces $H\Lambda^{k-1}(\Omega)$ and $H\Lambda^{k}(\Omega)$ 
by the finite element spaces $V_h^{k-1}$ and $V_h^k$.
More precisely, we are searching for a triple $(\tilde{\sigma}_h,\tilde{u}_h, \tilde{p}_h) \in V_h^{k-1} \times V_h^k \times \mfh_h^k$ such that 
\begin{align}
\notag \LRa{\tilde{\sigma}_h, \tau} - \LRa{d \tau, \tilde{u}_h}  &= 0 , & &   \tau \in V_h^{k-1}, \\
\label{hodge-mixed-1-h}\LRa{d\tilde{\sigma}_h, v} + \LRa{d\tilde{u}_h, dv} + \LRa{\tilde{p}_h, v} &= \LRa{f, v}, & &   v \in V_h^k, \\
\notag \LRa{\tilde{u}_h, q} &= 0, & &   q \in \mfh_h^k,
\end{align}
where the space $\mfh_h^k$, approximating the harmonic forms, is given by 
\[
\mfh_h^k = \{ v \in V_h^k \;:\; d v = 0 \text{ and }\LRa{v, d \tau} = 0 \text{ for all }\tau \in V_h^{k-1} \} .
\]
The goal of this paper is to study certain perturbations of the system \eqref{hodge-mixed-1-h}. Therefore, we have used an 
unconventional notation for the solution of \eqref{hodge-mixed-1-h} in order to be in position 
to use the more standard notation, $(\sigma_h, u_h, p_h)$, for the solution of the perturbed problems.

The exterior derivative appearing in the method is the exact operator $d$, restricted to the spaces $V_h^{k-1}$ and $V_h^k$,
while  no  $d^*$ operator appears. Instead, an approximation of $d^*$ is implicitly defined by the system \eqref{hodge-mixed-1-h}.
More precisely, the operator $d_h^*: V_h^k \to V_h^{k-1}$ is defined by the first equation of 
the system \eqref{hodge-mixed-1-h}, i.e.,
\begin{equation}\label{dh*}
\LRa{ d_h^* u, \tau} = \LRa{u, d\tau}, \quad u \in V_h^k, \, \tau \in V_h^{k-1}.
\end{equation}
In fact, just as we have explained for the special discrete problem \eqref{mixed-Lap} above,
the continuity requirements of the spaces $V_h^{k-1}$ will in general have the effect  the operator $d_h^*$ is nonlocal,
in contrast to the continuous case where $d^*$ is a local operator. Motivated by this our purpose in this paper is to
construct perturbations of the standard mixed methods which are converging, but also have the property that the corresponding operator 
$d_h^*$ is local. We will achieve this by replacing the $L^2$ inner product $\LRa{ d_h^* u, \tau}$ in \eqref{dh*} by a proper
approximation,
and by choosing the spaces $V_h^{k-1}$ and $V_h^k$ carefully. 


If $\{\mc{T}_h\}$ is a family of simplicial meshes, as described for example in \cite[Section 5]{AFW06}, 
then the spaces $V^k_h$ are taken from
two main families.  Either $V^k_h$ is of the form
$\mc{P}_r\Lambda^k(\mc{T}_h)$, consisting of all elements of
$H\Lambda^k(\Omega)$ which restrict to polynomial $k$-forms of degree
at most $r$ on each simplex $T$ in the partition $\mc{T}_h$, or
$V^k_h=\mc{P}^-_r\Lambda^k(\mc{T}_h)$, which is a space which sits
between $\mc{P}_r\Lambda^k(\mc{T}_h)$ and $\mc{P}_{r-1}\Lambda^k(\mc{T}_h)$. 
In  addition, both spaces have the property that the elements have continuous traces on each simplex in $\Delta_{n-1}(\mc{T}_h)$, and as a consequence they are subspaces of $H \Lm^k(\Omega)$.
Here we adopt the notation that $\Delta_k(\mc{T}_h)$ denotes the set of all the $k$--dimensional subsimplexes 
of the triangulation $\mc{T}_h$. 
The simplest stable discretization of the Hodge Laplace problem is obtained by choosing both spaces $V_h^{k-1}$ and $V_h^k$
to be the classical Whitney forms, i.e., we take $V_h^{k-1} = \mc{P}^-_1\Lambda^{k-1}(\mc{T}_h)$
and $V_h^{k} = \mc{P}^-_1\Lambda^{k}(\mc{T}_h)$. For the space  $\mc{P}^-_1\Lambda^{k}(\mc{T}_h)$ the degrees of freedom 
are simply the integrals of the traces over each element of  $\Delta_k(\mc{T}_h)$. The corresponding degrees of freedom for 
the corresponding linear space, $\mc{P}_1\Lambda^{k}(\mc{T}_h)$, is the corresponding integrals 
over each element of  $\Delta_k(\mc{T}_h)$ against all scalar linear test functions, cf. \cite[Theorem 4.10]{AFW06}.
As we will see below, this extra local freedom, represented by linear test functions on the $k$-dimensional subsimplexes, 
will be  crucial for our construction of local methods below.

%
%

\section{Abstract error analysis and variational crimes}\label{sec-abstract}
Discussion of  finite element methods and variational crimes in various settings 
is standard, and can be found in textbooks like \cite{MR3097958,Brenner-Scott-book}.
We will find it useful to base our analysis below on some abstract error estimates in the general setting of Hilbert complexes. In this respect our discussion in this section resembles parts of the theory presented in \cite{MR2915563}.
However, 
the main result here, cf. Theorem~\ref{thm:main} below, is targeted more directly to the applications later in the paper.
Our notation and set--up are basically taken from \cite[Chapter 3]{AFW10}.

A closed Hilbert complex $(W, d)$ consists of a sequence of Hilbert spaces $W^k$ with index $k$ and a sequence of closed, densely-defined linear operators $d^k : W^k \ra W^{k+1}$ such that $d^{k+1} \circ d^k = 0$. The sequence of operators $d$ is called a differential. A Hilbert subcomplex of $(W,d)$ is a Hilbert complex $(\bar{W}, \bar{d})$ such that $\bar{W}^k$ is a subspace of $W^k$ and $\bar{d}^k = d^k|_{\bar{W}^k}$ for each $k$. The domain complex of a closed Hilbert complex $(W,d)$ is the Hilbert subcomplex $(V,d)$ such that $V^k \subset W^k$ is the domain of $d^k$ for all $k$. 
We use $\LRa{\cdot, \cdot}$ and $\| \cdot \|$ to denote the inner product and the corresponding norm on $W^k$, respectively, but we omit index $k$ since it is usually clear from the  context. Similarly, we use $\LRa{\cdot, \cdot}_V$ to denote the inner product 
\algn{ \label{eq:V-inner}
\LRa{\omega, \omega'}_V := \LRa{\omega, \omega'} + \LRa{d^k \omega, d^k \omega'}, \qquad \omega, \omega' \in V^k
}
and $\| \cdot \|_V$ is the associated norm. 
The dual complex $(W, d^*)$, associated to $(W,d)$, is the Hilbert complex with same $W^k$ as Hilbert spaces and $d_{k+1}^* : W^{k+1} \ra W^k$, the adjoint of $d^k$, as differential. The $d^*$ is also called the coderivative of $d$. We say that $d^k$ is closed and densely-defined if the range of $d^k$ is closed in $W^{k+1}$ and the domain of $d^k$ is dense in $W^k$. Furthermore, the abstract space corresponding to the harmonic forms is given by 
\[
\mfh^k = \{ \omega \in V^k \;:\; d \omega = 0 \text{ and }\LRa{\omega, d \eta} = 0 \text{ for all }\eta \in V^{k-1} \}.
\]
In an abstract Hilbert complex, for a given $f \in W^k$ with $f \perp \mfh^k$, a variational mixed form of the Hodge Laplace problem \eqref{hodge-Lap-strong} is to find $(\sigma, u, p) \in V^{k-1} \times V^k \times \mfh^k$ such that 
\begin{align}
\notag \LRa{\sigma, \tau} - \LRa{d \tau, u}  &= 0 , & &   \tau \in V^{k-1}, \\
\label{eq:cont-eq} \LRa{d\sigma, v} + \LRa{du, dv} + \LRa{p, v} &= \LRa{f, v}, & &   v \in V^k, \\
\notag \LRa{u, q} &= 0, & &   q \in \mfh^k .
\end{align}
To discretize \eqref{eq:cont-eq}, we assume that $(V_h, d)$ is a family of finite dimensional subcomplexes 
parametrized by a discretization parameter $h \in (0,1]$. So 
$V_h^k \subset V^k$ and $d(V_h^{k-1}) \subset V_h^k$. Furthermore, we assume that the discretization 
is stable in the sense that there exist uniformly bounded cochain projections, cf. 
\cite[Section 3.3]{AFW10}. 
As a consequence, if we define
\[
\mfz_h^k = \{ \omega \in V_h^{k} \; : \; d \omega = 0 \}, \quad \text{and } 
\mfh_h^k = \{ \omega \in V_h^k \;:\; d \omega = 0, \; \LRa{\omega, d \tau} = 0 \quad    \tau \in V_h^{k-1} \},
\] 
then a discrete Poincar\'{e} inequality holds, i.e., there exists $c_P>0$, independent of $h$,  such that 
\algn{ \label{eq:disc-poincare}
\| v \| \leq c_P \| d v \|, \qquad v \in \mfz_h^{k \perp}.
}
Here $\mfz_h^{k \perp}$ is the orthogonal complement of $\mfz_h^k$ in $V_h^k$, cf.  \cite{AFW10}. 

Let $\mc{X}_h^k = V_h^{k-1} \times V_h^{k} \times \mfh_h^k$. The discrete problem corresponding to \eqref{eq:cont-eq} 
is to find $(\tilde{\sigma}_h, \tilde{u}_h, \tilde{p}_h) \in \mc{X}_h^k$ such that a system of the form 
\eqref{hodge-mixed-1-h} holds. Alternatively, this system can be written 
\algn{
\label{eq:cont-B-eq} \mathcal{B} (\tilde{\sigma}_h, \tilde{u}_h, \tilde{p}_h; \tau, v, q) &= \LRa{f, v}, & &   
(\tau, v, q) \in \mc{X}_h^k, 
}
where the bilinear form is give by
\algn{ \label{eq:B}
\mathcal{B} (\tau, v, q; \tau', v', q') = \LRa{\tau, \tau'} - \LRa{d \tau', v} + \LRa{d \tau, v'} + \LRa{dv, dv'} + \LRa{q, v'} - \LRa{q', v}.
}
If we let $\| (\tau, v, q) \|_{\mc{X}} := \| \tau \|_{V} + \| v \|_{V} + \| q \|$,
then stability and error estimates for the discrete approximations  are given in
\cite[Theorem~3.9]{AFW10}. An estimate of the form 
\algn{ \label{eq:tilde-approx}
\| (\sigma, u, p) - (\tilde{\sigma}_h, \tilde{u}_h, \tilde{p}_h) \|_{\mc{X}} \quad \lesssim \inf_{(\tau,v,q) \in \mc{X}_h^k} \| (\sigma, u, p) - (\tau, v, q) \|_{\mc{X}} + \mc{E}_h(u) 
}
holds.
Here, and below, the notation  $X\lesssim Y$ is used to state that  $X \leq C Y$, with a   
constant $C>0$ independent of the discretization parameter $h$. The extra error term 
$\mc{E}_h(u)$ appears as a consequence of the fact that the space of discrete harmonic forms, $\mfh_h^k$,
is not a subspace of $\mfh^k$. In fact, this term will be zero if there are no nontrivial harmonic forms. 

To achieve a numerical method which results in a local $d_h^*$ operator in the concrete settings below,
we will consider discrete problems with a perturbed bilinear form $\mathcal{B}_h$. 
More precisely, the first equation of \eqref{hodge-mixed-1-h} is modified.  To define $\mc{B}_h$, we suppose that there is a bilinear form $\LRa{\cdot, \cdot}_h$ satisfying the following assumption:
\begin{itemize}
 \item[\bf (A)] $\LRa{\cdot, \cdot}_h$ is a symmetric bounded coercive bilinear form on $V_h^{k-1} \times V_h^{k-1}$ such that the norm $\| \tau \|_h:= \LRa{\tau , \tau}_h^{1/2}$ is equivalent to $\| \tau \|$ for $\tau \in V_h^{k-1}$ with constants independent of $h$. 
\end{itemize}
Explicit examples of $\LRa{\cdot,\cdot}_h$ on simplicial and cubical meshes, leading to a
local $d_h^*$ operator,  will be given in the next two sections.
We now define $\mathcal{B}_h :\mc{X}_h^k \times \mc{X}_h^k \ra \R$ by 
\algn{ \label{eq:Bh}
\mc{B}_h (\tau, v, q ; \tau', v', q')  = \mc{B}(\tau, v, q ; \tau', v', q') + 
\LRa{\tau, \tau'}_h - \LRa{\tau, \tau'}
}
and consider the problem to find $(\sigma_h, u_h, p_h) \in \mc{X}_h^k$ such that 
\algn{ \label{eq:Bh-eq}
\mc{B}_h (\sigma_h, u_h, p_h ; \tau, v, q) = \LRa{f, v}, \qquad   (\tau, v, q) \in \mc{X}_h^k .
}
Let us define the norm $\tnorm{(\tau, v, q )}$ for $(\tau, v, q) \in \mc{X}_h^k$ by
\algns{
\tnorm{(\tau, v, q )} := \LRp{ \| \tau \|_{h}^2 + \| d \tau \|^2 + \| v \|_{V}^2 + \| q \|^2 }^{\half} . 
}
From the assumption {\bf (A)} it is easy to see that $\tnorm{\cdot}$ is equivalent to $\| \cdot \|_{\mc{X}}$ 
in $\mc{X}_h^k$. Due to {\bf (A)} and the discrete Poincar\'{e} inequality \eqref{eq:disc-poincare}, there exists a positive constant, again denoted by ${c}_P$, that 
\begin{align} \label{eq:poincare}
\| \rho \|_{h}  \leq {c}_P \| d \rho \|, \quad \rho \in \mfz_h^{k-1 \perp}. 
\end{align}
\begin{theorem} \label{thm:stability} Suppose that $\mathcal{B}_h$ is defined as in \eqref{eq:Bh} with $\LRa{\cdot, \cdot}_h$ satisfying {\bf (A)}. Then $\mathcal{B}_h$ satisfies
\algn{ \label{eq:Bh-infsup}
\inf_{0 \not = (\tau, v, q) \in \mc{X}_h^k } \sup_{0 \not = (\tau', v', q') \in \mc{X}_h^k } \frac{ \mathcal{B}_h (\tau, v, q ; \tau', v', q')}{\tnorm{(\tau, v, q )} \tnorm{( \tau', v', q' )}} \gtrsim 1 .
}
\end{theorem}
The proof of this theorem is completely analogous to the proof of Theorem~3.2 in \cite{AFW10},
using  the discrete Poincar\'e inequality \eqref{eq:poincare}, so we do not prove it here.

By Theorem~\ref{thm:stability}, \eqref{eq:Bh-eq} has a unique solution $(\sigma_h, u_h, p_h) \in \mc{X}_h^k$, and 
to show the convergence of the perturbed method 
we only need to estimate $\|(\sigma_h - \tilde{\sigma}_h, u_h - \tilde{u}_h, p_h - \tilde{p}_h)\|_{\mc{X}}$. 
%
The equivalence of $\tnorm{\cdot}$ and $\| \cdot \|_{\mc{X}}$, \eqref{eq:Bh-eq} and \eqref{eq:Bh-infsup},  lead us to 
\algn{ 
\notag &\| (\sigma_h - \tilde{\sigma}_h, u_h - \tilde{u}_h, p_h - \tilde{p}_h) \|_{\mc{X}} \\
\notag &\quad \lesssim \sup_{(\tau, v, q) \in \mc{X}_h^k}  \frac{\mathcal{B}_h (\sigma_h - \tilde{\sigma}_h, u_h - \tilde{u}_h, p_h - \tilde{p}_h ; \tau, v, q)}{\tnorm{ (\tau, v, q)}}  \\
\label{eq:strang} &\quad = \sup_{(\tau, v, q) \in \mc{X}_h^k} \frac{\LRa{f, v} - \mathcal{B}_h (\tilde{\sigma}_h, \tilde{u}_h, \tilde{p}_h ; \tau, v, q)}{\tnorm{ (\tau, v, q)}} \\
\notag &\quad = \sup_{\tau \in V_h^{k-1}} \frac{\LRa{\tilde{\sigma}_h, \tau}_h - \LRa{\tilde{\sigma}_h, \tau}}{\| \tau \|_V} .
}

This shows that convergence of $\| (\sigma_h - \tilde{\sigma}_h, u_h - \tilde{u}_h, p_h - \tilde{p}_h) \|_{\mc{X}}$ is related to the consistency error from the discrete bilinear form $\LRa{\cdot, \cdot}_h$. To have a consistency error estimate, we need another assumption for $\LRa{\cdot, \cdot}_h$:
\begin{itemize}
\item[{\bf (B)}] There exist discrete subspaces $W_h^{k-1} \subset W^{k-1}$ and $\tilde{V}_h^{k-1} \subset V_h^{k-1}$ such that 
\algn{ \label{eq:B-cond}
\LRa{\tau, \tau_0} = \LRa{ \tau, \tau_0}_h, \qquad \tau \in \tilde{V}_h^{k-1}, \tau_0 \in W_h^{k-1},
}
and a linear map $\Pi_h : V_h^{k-1} \ra \tilde{V}_h^{k-1}$ such that $ d \Pi_h \tau = d \tau $, $\| \Pi_h \tau \| \lesssim \| \tau \|$, and 
\algn{ \label{eq:B-cond2}
  \LRa{ \Pi_h \tau, \tau_0 }_h =        \LRa{ \tau , \tau_0}_h, \quad  \tau_0 \in W_h^{k-1}. 
}
\end{itemize}
Note that if \eqref{eq:B-cond} holds with $\tilde{V}_h^{k-1}=V_h^{k-1}$, then all other conditions are satisfied with $\Pi_h$ as the identity map. 

The error bound of $\| (\sigma - \sigma_h, u - u_h, p - p_h) \|_{\mc{X}}$ follows from the estimate of $\| (\sigma_h - \tilde{\sigma}_h, u_h - \tilde{u}_h, p_h - \tilde{p}_h) \|_{\mc{X}}$ obtained in the theorem below. 
\begin{theorem} \label{thm:main} Suppose that $\mc{B}_h$ is given as in \eqref{eq:Bh} with $\LRa{\cdot, \cdot}_h$ satisfying conditions {\rm \bf (A)} and {\bf(B)}. 
Then the solution of \eqref{eq:Bh-eq},  $(\sigma_h, u_h, p_h)$, satisfies
\algn{ \label{eq:err-estm}
\| (\sigma_h - \tilde{\sigma}_h, u_h - \tilde{u}_h, p_h - \tilde{p}_h) \|_{\mc{X}} \lesssim \| \sigma - P_{W_h} \sigma \| + \| \sigma - \tilde{\sigma}_h \|,
}
where $P_{W_h}$ is the $W$-orthogonal projection onto $W_h^{k-1}$.
\end{theorem}
\begin{proof} 
By \eqref{eq:strang}, it suffices to show 
\algn{\label{eq:D-estm}
| \LRa{ \tilde{\sigma}_h , \tau} - \LRa{ \tilde{\sigma}_h , \tau}_h | \;\lesssim \LRp{ \| \sigma - P_{W_h} \sigma \| + \| \sigma - \tilde{\sigma}_h \| } \| \tau \|_V .
}
We first observe that since $d \Pi_h \tau' = d \tau' $ we have
$\LRa{ \tilde{\sigma}_h, \tau'} = \LRa{\tilde{\sigma}_h, \Pi_h \tau'}$ for $\tau' \in V_h^{k-1}$.  Using this equality, we have 
\algn{ 
\notag \LRa{ \tilde{\sigma}_h , \tau} - \LRa{ \tilde{\sigma}_h , \tau}_h &= \LRa{ \tilde{\sigma}_h , \Pi_h \tau} - \LRa{ \tilde{\sigma}_h , \tau}_h \\
\label{eq:D-id} &= \LRa{ \tilde{\sigma}_h , \Pi_h \tau} -  \LRa{ \tilde{\sigma}_h , \Pi_h \tau}_h + \LRa{ \tilde{\sigma}_h , \Pi_h \tau - \tau}_h \\
\notag &= \LRa{ \tilde{\sigma}_h - P_{W_h} {\sigma}, \Pi_h \tau} -  \LRa{ \tilde{\sigma}_h - P_{W_h} {\sigma} , \Pi_h \tau}_h \\
\notag &\quad + \LRa{ \tilde{\sigma}_h - P_{W_h} {\sigma} , \Pi_h \tau - \tau}_h , 
}
where, to get the last equality, we used \eqref{eq:B-cond} for the first two terms and \eqref{eq:B-cond2} for the last term, respectively.
By the Cauchy--Schwarz inequality and the boundedness of $\LRa{\cdot, \cdot}_h$ in {\bf (A)}, we have 
\algn{ \label{eq:D-subestm1}
| \LRa{ \tilde{\sigma}_h - P_{W_h} {\sigma}, \Pi_h \tau} -  \LRa{ \tilde{\sigma}_h - P_{W_h} {\sigma} , \Pi_h \tau}_h | &\lesssim (\| \tilde{\sigma}_h - \sigma \| + \| \sigma - P_{W_h}\sigma \|) \| \tau \| .
}
 
A similar argument with $\| \Pi \tau - \tau \| \lesssim \| \tau \|$ from the boundedness of $\Pi_h$, gives 
\algn{ \label{eq:D-subestm2}
| \LRa{ \tilde{\sigma}_h - P_{W_h} {\sigma} , \Pi_h \tau - \tau}_h | &\lesssim \| \tilde{\sigma}_h - P_{W_h} {\sigma} \| \| \Pi_h \tau - \tau \| \\
\notag &\lesssim (\| \tilde{\sigma}_h - \sigma \| + \| \sigma - P_{W_h} {\sigma} \|) \| \tau \| .  
}
Then, \eqref{eq:D-estm} follows from \eqref{eq:D-id}, the triangle inequality, and the estimates \eqref{eq:D-subestm1} and \eqref{eq:D-subestm2}. This completes the proof.
\end{proof}

In summary, we have presented perturbation results for mixed approximations of
abstract Hodge Laplace problems with sufficient conditions for well-posedness and error estimates. 
If the method is based on a standard  mixed formulation of the form \eqref{hodge-mixed-1-h}, which is stable, then the extra error 
introduced by the modification of the bilinear form $\mc{B}$ into $\mc{B}_h$, cf. \eqref{eq:B} and \eqref{eq:Bh}, is controlled by 
Theorem~\ref{thm:main} above.
Hence, the extra conditions to check are conditions  {\rm \bf (A)} and {\bf(B)}. 

\section{The simplicial case}

In this section we apply the abstract framework in the previous section to mixed Hodge Laplace problems of the de Rham complex on simplicial meshes. We let $\Omega$ be a bounded polyhedral domain in $\R^n$. Recall that the de Rham complex on $\Omega$ is the Hilbert complex $(W,d)$ with $W^k= L^2 \Lm^k(\Omega)$, $0 \leq k \leq n$, and with corresponding domain complex $(V,d)$, where  $V^k= H \Lm^k(\Omega)$. Here, 
$d= d^k : H \Lm^k(\Omega) \ra L^2 \Lm^{k+1}(\Omega)$ is the exterior derivative. 
 Let $\{ \mc{T}_h\}$ be a family of shape-regular simplicial meshes of $\Omega$, indexed by the parameter $h = \max \{ \diam T \: : \: T \in \mc{T}_h \, \}.$
Associated to the mesh $\mc{T}_h$ there are basically two families of finite element spaces of  differential forms, $\mc{P}_r \Lm^k(\mc{T}_h)$ and 
$\mc{P}_r^- \Lm^k(\mc{T}_h)$ for $0 \leq k \leq n$, where $r$ is the local polynomial degree.

In our discussion below  we will study concrete realizations of discretizations of the form \eqref{eq:Bh-eq}, where the discrete spaces 
$V_h^{k-1}$ and $ V_h^k$ are chosen as $\mc{P}_1 \Lm^{k-1} (\mc{T}_h)$ and 
$\mc{P}_1^- \Lm^k (\mc{T}_h)$, respectively.
In other words, we are combining  the lowest order finite element spaces of the two basic families. The exterior derivative $d$ maps 
$\mc{P}_1 \Lm^{k-1} (\mc{T}_h)$ into $\mc{P}_1^- \Lm^k (\mc{T}_h)$.
In fact, this pair leads to a stable discretization of the corresponding Hodge Laplace problem in the form of \eqref{hodge-mixed-1-h}. Furthermore, 
the right-hand side of \eqref{eq:tilde-approx} is of order $O(h)$ under the assumption that the solution is sufficiently regular, cf.  \cite[Section 7.6]{AFW06} or \cite[Chapter 5]{AFW10}. The corresponding discrete coderivative $d_h^*$, defined by \eqref{dh*}, is  a map from
$\mc{P}_1^- \Lm^k (\mc{T}_h)$ to $\mc{P}_1 \Lm^{k-1} (\mc{T}_h)$.  However, as discussed  in Sections 1 and 2 above, this operator will be nonlocal 
as a consequence of the continuity properties of the space $\mc{P}_1 \Lm^{k-1} (\mc{T}_h)$. Therefore, 
to achieve a local $d_h^*$ operator we will follow the approach of Section 3 above, and modify the inner product 
on $\mc{P}_1 \Lm^{k-1} (\mc{T}_h)$, cf. \eqref{eq:Bh}. More precisely, we will replace the $L^2$ inner product $\LRa{\cdot, \cdot}$ on 
$\mc{P}_1 \Lm^{k-1} (\mc{T}_h)$ by a modified inner product $\LRa{\cdot, \cdot}_h$. The main purpose of this 
modification is to obtain a local coderivative $d_h^*$,  mapping $\mc{P}_1^- \Lm^k (\mc{T}_h)$ to $\mc{P}_1 \Lm^{k-1} (\mc{T}_h)$, 
defined by 
\begin{equation}\label{dh*mod}
\LRa{ d_h^* u, \tau}_h = \LRa{u, d\tau}, \quad u \in \mc{P}_1^- \Lm^k (\mc{T}_h), \, \tau \in \mc{P}_1 \Lm^{k-1} (\mc{T}_h).
\end{equation}
To apply the convergence theory of Section 3 we need to verify that the stability condition  {\bf (A)} and 
the consistency condition {\bf (B)} hold. In the present case we will verify condition  {\bf (B)} with the space 
$\tilde{V}_h^{k-1}$ taken to be $V_h^{k-1}$. Hence, to verify this condition we only need to show that condition 
\eqref{eq:B-cond} holds for a proper space  $W_h^{k-1} \subset L^2\Lm^{k-1}(\Omega)$. Furthermore, to preserve the linear convergence of the method
the space $W_h^{k-1}$ should have the property that the $L^2$ error of the orthogonal projection  onto $W_h^{k-1}$ 
is of order $O(h)$, cf. Theorem~\ref{thm:main}. In fact, throughout the discussion of this section we will take 
$W_h^{k-1}$ to be the space of piecewise constant forms, i.e., 
\begin{equation}\label{eq:Wh-simplicial}
W_h^{k-1} := \{ \tau \in L^2 \Lm^{k-1} (\Omega)\, | \,  \tau|_T \in \mc{P}_0 \Lm^{k-1} (T) \text{ for all } T \in \mc{T}_h \},
\end{equation}
and as a consequence the desired accuracy of the projection is achieved.

Instead of discussing how to construct the modified inner product $\LRa{ d_h^* u, \tau}_h$ on one specific space, 
$\mc{P}_1 \Lm^{k -1}(\mc{T}_h)$, we will  consider the construction of such modified inner products on all the spaces 
of the form $\mc{P}_1 \Lm^{k }(\mc{T}_h)$, where $0 \le k \le n$. We recall that 
the space $\mc{P}_r \Lm^k (\R^n)$, i.e., the space of polynomial $k$-forms of degree $r$,  consists of all polynomials of degree $r$ with values in $\Alt^k(\R^n)$,
and its dimension is given by 
\algns{
\dim \mc{P}_r \Lm^k (\R^n) = \pmat{n+r \\r} \pmat{n \\k} .
}
Furthermore, 
an element $u$ of $\mc{P}_1 \Lm^k (\R^n)$ is of the form 
\[
u(x) = a_0 + \sum_{j = 1}^n a_j x_j, \qquad a_j \in \Alt^k(\R^n).
\]
Hence, to determine $u$ on a simplex $T \in \mc{T}_h$ we need $(n+1)\pmat{n \\k}$ degrees of freedom. 
The standard degrees of freedom for the space $\mc{P}_1 \Lm^k (\mc{T}_h)$ is given by
\begin{equation}\label{DOFS}
\int_f \tr_f u \wedge v, \qquad v \in \mc{P}_1 \Lm^0 (f), f \in \Delta_k(\mc{T}_h),
\end{equation}
cf. \cite[Theorem 4.10]{AFW06}.
In fact, for any $f \in \Delta_k(\mc{T}_h)$ an element in $\mc{P}_1 \Lm^k (f)$ can be uniquely identified with an element in 
$\mc{P}_1 \Lm^0 (f)$ through the Hodge star operator on $f$. Therefore,
 the degrees of freedom given 
by \eqref{DOFS}, on a fixed $f \in \Delta_k(\mc{T}_h)$, determines the 
$\tr_f u$ uniquely. This means that 
\[ 
\dim \mc{P}_1 \Lm^k (\mc{T}_h) = (k+1) | \Delta_k(\mc{T}_h)|,
\]
where $| \Delta_k(\mc{T}_h)|$ is the cardinality of the set $\Delta_k(\mc{T}_h)$. Furthermore, the degrees of freedom 
given by \eqref{DOFS} can be replaced by any other set of degrees of freedom which determines $\tr_f u$ uniquely on each
$f \in \Delta_k(\mc{T}_h)$.

If $u \in \mc{P}_1 \Lm^k (\mc{T}_h)$ and $f \in \Delta_k(\mc{T}_h)$ then $\tr_f u$ is uniquely determined by 
$\tr_f u$ evaluated at each vertex of $f$.
In particular, if $f$ has vertices $x_0,x_1, \ldots ,x_k$, i.e., $f = [x_0,x_1, \ldots x_k]$, then $\tr_f u$ at vertex  $x_i$ 
is determined by the functional $\phi_{f,x_i}(u)$ given by 
\[
\phi_{f,x_i}(u) = u_{x_i} (x_0- x_i, \ldots, x_{i-1} - x_i, x_{i+1} - x_i, \ldots, x_k- x_i).
\]
In other words, at the point $x_i$ we apply the $k$-form $u$  to the $k$ vectors $x_j- x_i$, $j \neq i$, 
which all are tangential to $f$. By letting the index $i$ run from $0$ to $k$ we obtain $k+1$ functionals which determine $\tr_f u$ uniquely.
Furthermore, an element $u \in \mc{P}_1 \Lm^k (\mc{T}_h)$ is uniquely determined by the  degrees of freedom $\{ \phi_{f,x}(u) \}$,
where $(f,x)$ runs over all pairs such that $f \in \Delta_k(\mc{T}_h)$ and $x \in \Delta_0(f)$. Of course, this is again $(k+1) | \Delta_k(\mc{T}_h)|$
linearly independent degrees of freedom.

If $u \in \mc{P}_1 \Lm^k (\mc{T}_h)$ and $x_i$ is a vertex, i.e., $x_i \in \Delta_0(\mc{T}_h)$, then $u$ is not continuous at $x_i$.
In general, $u$ will have a separate value for each $T \in \mc{T}_h$ which touches  $x_i$. However, the value of $u$ at $x_i$,  taken in the simplex $T$,
is uniquely determined by the $\pmat{n \\ k} $ degrees of freedom given by $\phi_{f,x_i}(u)$ for all $f \in \Delta_k(T)$ such that 
$x_i \in \Delta_0(f)$. As a consequence, it follows that all the values of $u$ at $x_i$ are determined by $\phi_{f,x_i}(u)$
for $f \in \Delta_k(\mc{T}_h)$ with $x_i \in \Delta_0(f)$. In particular, if $\phi_{f,x}(u) = 0$ for a fixed $x \in \lap_0(\mc{T}_h)$ and all $f \in \lap_k(\mc{T}_h)$ such that $x \in \lap_0(f)$, then $u=0$ at $x$.

Given a set of degrees of freedom we can also define the corresponding dual basis $\{ \psi_{f,x} \}$ for the space $\mc{P}_1 \Lm^k (\mc{T}_h)$
defined by 
\begin{equation}\label{dual}
\phi_{g,y} (\psi_{f,x}) = \delta_{(f,x),(g,y)}, \qquad f,g \in \Delta_k(\mc{T}_h), \, x \in \Delta_0(f), \, y \in \Delta_0(g),
\end{equation}
with the obvious interpretation that $\delta_{(f,x),(g,y)} = 1$ if $(f,x) = (g,y)$ and zero otherwise. It is clear from the above observation that $\psi_{f,x} = 0$ at all $y \in \lap_0(\mc{T}_h)$ such that $y \not = x$.
In fact, the piecewise linear form $\psi_{f,x}$ has a simple explicit representation in terms of barycentric coordinates.
To see this, for $x_j \in \Delta_0(\mc{T}_h)$ we let $\lambda_j$ be the piecewise linear function determined by $\lambda_j(x_j)= 1$, while
$\lambda_j$ vanish on all other vertices. In other words, $\lambda_j$ corresponds to the barycentric 
coordinate associate the vertex $x_j$ for all $T \in \mc{T}_h$ such that $x_j \in \Delta_0(T)$, and it is extended by zero elsewhere. Note that the corresponding 
1-form, $d\lambda_j$, is piecewise constant
and vanish outside the macroelement $\Omega_{x_j}$. Here we use the notation that for any $f \in \Delta(\mc{T}_h)$, the associated 
macroelement $\Omega_f$ is given by
\[
\Omega_f = \bigcup \{ \, T \, | \, T \in \mc{T}_h, \, f \in \Delta(T) \, \}.
\]
In particular, we note that if $[x_j, x_i] \in \Delta_1(\mc{T}_h)$ then $d\lambda_j(x_j -x_i) = 1$.
On the other hand,  $d\lambda_j(x_l -x_i) = 0$ if $[x_i,x_l] \in \Delta_1(\mc{T}_h)$ with both endpoints different from  $x_j$.

Assume now that $f = [x_0, x_1, \ldots x_k] \in \Delta_k(\mc{T}_h)$. The corresponding functions
$\psi_{f,x_i}$, $i =0,1, \ldots ,k$
are given by 
\[
\psi_{f,x_i} = \lambda_i d\lambda_0 \wedge \cdots \wedge d\lambda_{i-1} \wedge d\lambda_{i+1} \wedge \cdots \wedge d\lambda_k,
\]
where $\wedge$ denotes the wedge product. The functions $\psi_{f,x_i}$ given above are obviously in $\mc{P}_1 \Lm^k (\mc{T}_h)$ and it is straightforward to check that they satisfy 
the conditions \eqref{dual}. Observe also that the basis functions $\psi_{f,x_i}$ have local support. In fact,  $\supp(\psi_{f,x_i}) \subset \Omega_{x_i}$.
The basis $\{ \psi_{f,x} \}$ for the space $\mc{P}_1 \Lm^k (\mc{T}_h)$, just introduced, is related to point values of traces
via the dual relation \eqref{dual}. Furthermore, the modified inner product 
$\LRa{\cdot, \cdot}_{h}$ will also be defined by point values. In fact, 
the modified inner product will be constructed such that the matrix $\LRa{\psi_{f,x_i}, \psi_{g,x_j}}_{h}$
is block diagonal,
and this  is the key property we will use below to show that the constructed 
coderivative $d_h^*$ is local.

To define the modified  inner product $\LRa{\cdot, \cdot}_h$ on $\mc{P}_1 \Lm^k (\mc{T}_h)$ we first recall 
that if $T = [x_0,x_1, \ldots ,x_n] \in \mc{T}_h$, then the identity 
\[
\int_T u(x) \, dx = \frac{ |T|}{n+1} \sum_{i= 0}^n  u(x_i)
\]
holds for all linear and scalar valued functions $u$. Here $|T|$ denotes the volume of $T$. Therefore, the bilinear form $\LRa{\cdot, \cdot}_{h,T}$,
given by 
\begin{equation}\label{quadrature-s}
\LRa{u, v}_{h,T} = \frac{ |T|}{n+1} \sum_{x\in \Delta_0(T)}  \LRa{u_{x}, v_{x}}_{\Alt},
\end{equation}
defines an inner product on $\mc{P}_1\Lm^k(T)$ which is exactly equal to the inner product on $L^2\Lm^k(T)$ for 
$u \in \mc{P}_1\Lm^k(T)$ and $v \in \mc{P}_0\Lm^k(T)$. As a consequence, if we define $\LRa{\cdot, \cdot}_{h}$ by
\begin{equation}\label{mod-inner}
\LRa{u, v}_{h} = \sum_{T \in \mc{T}_h} \LRa{u, v}_{h,T}, \qquad u,v \in \mc{P}_1 \Lm^k (\mc{T}_h),
\end{equation}
then this is an inner product on $\mc{P}_1 \Lm^k (\mc{T}_h)$ which is identical to the standard $L^2$ inner product 
if $u \in \mc{P}_1 \Lm^k (\mc{T}_h)$ and $v \in W_h^k$, cf. \eqref{eq:Wh-simplicial}. 
Furthermore, it follows from standard scaling arguments and shape regularity that the inner product 
$\LRa{\cdot, \cdot}_{h}$ is equivalent to the standard $L^2$ inner product on $\mc{P}_1 \Lm^k (\mc{T}_h)$,
with constants independent of $h$.

We can summarize the discussion so far as follows.

\begin{theorem} \label{thm:P1-conv}
For $1 \le k \le n$ let 
$V_h^{k-1} = \mc{P}_1 \Lm^{k-1} (\mc{T}_h)$ and  $V_h^k = \mc{P}_1^- \Lm^k (\mc{T}_h)$. Furthermore,  let
the bilinear form $\mc{B}_h$ be defined as in \eqref{eq:Bh} with $\LRa{\cdot,\cdot}$ being the appropriate $L^2$ inner products
and the modified inner product $\LRa{\cdot,\cdot}_h$ on $V_h^{k-1}$ defined as in 
\eqref{mod-inner}. Then the stability condition {\bf (A)} and the consistency condition {\bf (B)} holds, where $W_h^{k-1}$ is given as in \eqref{eq:Wh-simplicial}.
\end{theorem}

As we have already observed above the solution $(\sigma_h, u_h,p_h)$ of the problem \eqref{eq:Bh-eq}, 
with the set up given in the theorem above, will in general converge to the corresponding exact solution 
of the Hodge Laplace problem. This is just a consequence of the estimate \eqref{eq:tilde-approx} combined 
with Theorem~\ref{thm:main}. Furthermore, under the appropriate regularity assumptions on the exact solution 
the convergence will be linear with respect to the mesh size $h$, i.e., the error will be $O(h)$.

Next, we will consider the operator $d_h^*$ defined by \eqref{dh*mod}, and show that this operator is indeed a local operator.
This will basically follow from the fact that the matrix $\LRa{\psi_{f,x_i}, \psi_{g,x_j}}_h$ is block diagonal. 

\begin{theorem}\label{local-simplicial}
For $1 \le k \le n$ let 
$d_h^*: \mc{P}_1^- \Lm^k (\mc{T}_h) \to \mc{P}_1 \Lm^{k-1} (\mc{T}_h)$ be the operator defined by \eqref{dh*mod}. This operator is local. More precisely, 
for any vertex $x_i \in \Delta_0(\mc{T}_h)$ the values $(d_h^* u)_{x_i}$ only depends on $u$ restricted to the macroelement $\Omega_{x_i}$.
 \end{theorem}  
 
 \begin{proof}
 For any $u \in \mc{P}_1^- \Lm^k (\mc{T}_h)$ we can express $d_h^*u$ in terms of the basis functions $\psi_{f,x_i}$
 in $\mc{P}_1 \Lm^{k-1} (\mc{T}_h)$, i.e.,
 \[
 d_h^*u = \sum_{(f,x_i)} c_{f,x_i} \psi_{f,x_i}, \qquad c_{f,x_i} \in \R,
 \]
where the sum runs over all pairs $(f,x_i)$ such that $f \in \Delta_{k-1}(\mc{T}_h)$ and $x_i \in \Delta_0(f)$.
Furthermore, the matrix $\LRa{\psi_{f,x_i}, \psi_{g,x_j}}_{h}$ is block diagonal, where the blocks are
of the form $\LRa{\psi_{f,x_i}, \psi_{g,x_i}}_{h}$, i.e., they  correspond to the vertices  $x_i$ in  $\Delta_0(\mc{T}_h)$.
Therefore, 
if we fix a vertex $x_i$, then all the coefficients of the form $c_{f,x_i}$ is determined by the subsystem of \eqref{dh*mod}
of the form 
\begin{equation}\label{system-x_0}
\sum_f c_{f,x_i} \LRa{\psi_{f,x_i}, \psi_{g,x_i}}_{h} = \LRa{u, d\psi_{g,x_i}},
\end{equation}
where $f$  and $g$ runs over all elements of $\Delta_{k-1}(\mc{T}_h)$ which contains the vertex $x_i$. However, this represents a square positive definite system which determines
the coefficients $c_{f,x_i}$ uniquely, and hence all the values $(d_h^* u)_{x_i}$. Finally, since the support of the basis functions $\psi_{g,x_i}$ is contained in $\Omega_{x_i}$ it follows that  
the right 
hand side of \eqref{system-x_0} only depends on $u$ restricted to $\Omega_{x_i}$.
\end{proof}
It follows from the proof above that the coefficients $c_{f,x_i}$ can be computed from the local systems \eqref{system-x_0}.  When $x_i$ runs over all the vertices of the mesh, these matrices 
represent the diagonal blocks of the full matrix
$\LRa{\psi_{f,x_i}, \psi_{g,x_j}}_h$. In fact, the elements of the block associated the vertex $x_i$   can be explicitly expressed in terms 
of the volumes of the $n$ simplexes $T$ touching $x_i$, the volumes of $f$ and $g$, and the principal angles between $f$ and $g$. 

To see this, and to have the simplest notation, 
we perform this discussion of the matrix $\LRa{\psi_{f,x_i}, \psi_{g,x_j}}_h$ in the setting of $k$-forms instead of 
$(k-1)$-forms.
We fix a vertex in $\Delta_0(\mc{T}_h)$, and call it $x_0$. To compute the elements of the diagonal block of the matrix 
$\LRa{\psi_{f,x_i}, \psi_{g,x_j}}_h$, associated the vertex $x_0$, we let $f = [x_0, x_1, \cdots ,x_k] \in \Delta_k(\mc{T}_h)$.
If we assume that the vertices are ordered, such that the vectors $x_1 - x_0, \cdots ,x_k - x_0$ are positively oriented, then
\begin{equation}\label{rep-vol-f}
d\lambda_1 \wedge \cdots \wedge d\lambda_k = \frac{1}{k!\, |f|} \vol_f,
\end{equation}
when the forms are restricted to vectors which are tangential to $f$, cf. \cite[Section 4.1]{AFW06}.
Here $\vol_f$ denotes the standard volume form on $f$.
If $f$ and $g$ are two $k$--dimensional simplexes containing $x_0$ we then obtain  that 
\begin{align*}
\LRa{\psi_{f,x_0}, \psi_{g,x_0}}_h &= \sum_T \frac{|T|}{n+1} \LRa{(\psi_{f,x_0})_{x_0}, (\psi_{g,x_0})_{x_0}}_{\Alt}\\
&= \frac{1}{(n+1)(k!)^2} \sum_T \frac{|T|}{|f| \, |g|} \LRa{\vol_{f}, \vol_{g}}_{\Alt},
\end{align*}
where the sum is over all $T \in \mc{T}_h$ such that both $f$ and $g$ are in $\Delta_k(T)$. 
Furthermore, we assume that $\vol_{f}$ has been properly extended to 
a $k$-form
on  $\R^n$ such that \eqref{rep-vol-f} holds for all vectors and for $x \in T$.
However, the inner product $\LRa{\vol_f, \vol_g}_{\Alt}$ is related to the principal angles of the two $k$--dimensional subspaces of $\R^n$ containing 
$f$ and $g$, cf. for example \cite[Theorem~5]{MR1165446}. Therefore,  we have indeed obtained the desired representation of the elements of the matrix $\LRa{\psi_{f,x_j}, \psi_{g,x_j}}_h$.

\section{The cubical case} \label{sec:rect-eg}
In this section we present concrete realizations of the abstract framework studied in Section
\ref{sec-abstract} above
for approximations of the mixed Hodge Laplace problems on cubical meshes. 
Here a  cubical mesh $\mc{T}_h$ on the domain $\Omega$ is a mesh where each 
element is a  Cartesian product of intervals.

Mixed finite element methods with local coderivatives on cubical meshes have been studied by Wheeler and collaborators for the Darcy flow problems in two and three dimensions (i.e., $k=n$ and $n=2,3$), see  
\cite{Ingram-Wheeler-Yotov-2010,Wheeler-Xue-Yotov-2012}. 
In the two dimensional case 
the arguments are rather similar to the simplicial case. By choosing $V_h^{k-1} = V_h^1$ 
as the lowest order  Brezzi--Douglas--Marini space (${\rm BDM}_1$), cf. \cite{BDM85}, and piecewise constant functions for $V_h^k = V_h^2$, combined with an integration rule based on vertex values, a local coderivative $d_h^*$ is obtained.
However, the natural analog of this approach for $n=3$,
where the lowest order Brezzi--Douglas--Duran--Fortin space (${\rm BDDF}_1$, \cite{BDDF})
is chosen for $V_h^{k-1} = V_h^2$ and $V_h^k = V_h^3$ consist of piecewise constants,
will not lead to a corresponding local method. To overcome this problem Wheeler et al. replaced the standard ${\rm BDDF}_1$ space by an enriched space. The discussion in this section can be seen as a generalization of the discussion given in \cite{Ingram-Wheeler-Yotov-2010,Wheeler-Xue-Yotov-2012} to general $k$-forms in any dimension $n$. 
The most natural analogs of the ${\rm BDM}_1$ and ${\rm BDDF}_1$ spaces for the case of differential forms and higher space dimensions are
the $\mc{S}_1 \Lm^k(\mc{T}_h)$ spaces introduced by Arnold and Awanou in \cite{Arnold-Awanou-2014}, 
cf. the discussion given in the introduction of that paper. We will give a brief review of these spaces below.
However,  to 
obtain the  finite element spaces we need to obtain local approximations of the coderivatives, we  will enrich the finite element spaces 
$\mc{S}_1 \Lm^k(\mc{T}_h)$ to obtain a larger spaces,
which we will denote $\mc{S}_1^+ \Lm^k(\mc{T}_h)$.

For our discussion below we introduce some additional  notation.  
Recall the definition of the set $\Sigma(k)$ given in Section~\ref{prelim} above, i.e.,
$\s \in \Sigma(k)$ is an increasing  sequence with values $\s_i$, $1 \le i \le k$, such that 
\[
1 \le \s_1 < \s_2 < \ldots < \s_k \le n.
\]
The set $\Sigma(k)$ has $\pmat{n \\ k}$ elements.
We will use $\llbracket \s \rrbracket$ to denote the range of $\s$, i.e., 
\[
\llbracket \s \rrbracket = \{ \s_1, \s_2, \ldots , \s_k \} \subset \{ 1, 2, \ldots , n \},
\]
and $\s^*$ is the complementary sequence in $\Sigma(n-k)$ such that
\[
\llbracket \s \rrbracket \cup \llbracket \s^* \rrbracket = \{ 1, 2, \ldots , n \}.
\]
For each $\s \in \Sigma(k)$ we define 
$d x_\s = d x_{\s_1} \wedge \cdots \wedge d x_{\s_k}$ and the set $\{ dx_\s \, : \, \s \in \Sigma(k) \, \}$
is a basis of $\Alt^k(\R^n)$.
A differential $k$-form $u$ then admits the representation 
\algns{
u = \sum_{\s \in \Sigma(k)} u_{\s} d x_{\s},
}
where  the coefficients $u_{\s}$ are scalar functions on $\Omega$. Furthermore,  the exterior derivative $du$ 
can be expressed as 
\algns{
du = \sum_{\s \in \Sigma(k)} \sum_{i =1} ^n \pd_i u_{\s} d x_i \wedge d x_{\s} ,
}
if $\pd_i u_{\s}$ is well-defined as a function on $\Omega$. 
The Koszul operator $\kappa : \Alt^k(\R^n) \to \Alt^{k-1}(\R^n)$ is defined by contraction with the vector $x$, i.e.,
$(\kappa u)_x = u_x \lrcorner x$. As a consequence of the alternating property of $u$ it therefore follows that 
$\kappa \circ \kappa = 0$.
It also follows that 
\algns{ 
\kappa (dx_{\s}) = \kappa ( dx_{\s_1} \wedge \cdots \wedge dx_{\s_k} ) = \sum_{i=1}^k (-1)^{i+1} x_{\s_i} dx_{\s_1} \wedge \cdots \wedge \widehat{dx_{\s_i}} \wedge \cdots \wedge dx_{\s_k} ,
}
where $\widehat{dx_{\s_i}}$ means that the term $dx_{\s_i}$ is omitted. This definition is extended to the space of differential $k$-form on $\Omega$ by linearity, i.e., 
\algns{
\kappa u = \kappa { \sum_{\s \in \Sigma(k)} u_{\s} dx_{\s} } = \sum_{\s \in \Sigma(k)} u_{\s} \kappa( dx_{\s} ) .
}
If $f$ is an $(n-1)$-dimensional hyperplane of $\R^n$, obtained by fixing one coordinate, for example 
\[ 
f = \{\, x \in \R^n \, : \, x_1 = c \, \},
\]
then we can define the Koszul operator $\kappa_f$ for forms defined on $f$ by 
$(\kappa_f v)_x = v_x \lrcorner (x - x^f)$, where $x^f = (c,0, \ldots ,0)$. We note that the vector $x- x^f$ is in the tangent space 
of $f$ for $x \in f$.
Since $\tr_f (u \lrcorner (x - x^f)) = \tr_f u \lrcorner (x - x^f)$ for $x \in  f$ and $(\kappa u)_x = u_x \lrcorner (x-x^f) + u_x \lrcorner x^f$, we can conclude that
\begin{equation}\label{kos-tr}
\tr_f \kappa u = \kappa_f \tr_f u + \tr_f (u \lrcorner x^f).
\end{equation}
For a multi-index $\alpha$ of $n$ nonnegative integers, $x^{\alpha} = x_1^{\alpha_1} \cdots x_n^{\alpha_n}$. 
Hence, if $u$ is in $\mc{P}_r\Lm^k$, the space of polynomial $k$-forms of order $r$, then $u$ can be expressed as 
\[
u = \sum_{ \s \in \Sigma(k)}  u_\s dx_\s, \qquad u_\s = \sum_{|\alpha| \le r} c_{\alpha} x^{\alpha} \in \mc{P}_r,
\]
where $|\alpha | = \sum_i \alpha_i$. In other words, the coefficients $u_\s$ are ordinary real valued polynomials of degree less than or equal to $r$.
The corresponding tensor product space, $\mc{Q}_r\Lm^k$, consists of $k$-forms $u$ 
where the coefficients $u_\s$ is a tensor product of polynomials of degree less than or equal to $r$, i.e., 
\[
u_\s = \sum_{\alpha_i \le r, 1 \le i \le n} c_{\alpha} x^{\alpha} \in \mc{Q}_r.
\]
Denoting $\mc{H}_r \Lm^k$ the space of differential $k$-forms with homogeneous polynomial coefficients of degree $r$, we also have  the identity
\algn{ \label{eq:homotopy-formula}
(\kappa d + d \kappa ) u = (r + k) u , \qquad u \in \mc{H}_r \Lm^k ,
}
cf. \cite[Section 3]{AFW06}.

\subsection{The families $\mc{Q}_r^- \Lm^k$ and $\mc{S}_r \Lm^k$}
 Our discussion below relates two of the previously constructed families of finite element spaces 
with respect to cubical meshes, the $\mc{Q}_r^- \Lm^k$-family of \cite{Arnold-Boffi-Bonizzoni-2015}
and
the $\mc{S}_r \Lm^k$-family of \cite{Arnold-Awanou-2014}. Here the parameter $r \ge 1$
is related to the local polynomial degree, and for each $k$, $0 \le k \le n$, the spaces  
$\mc{Q}_r^- \Lm^k(\mc{T}_h)$ and $\mc{S}_r \Lm^k(\mc{T}_h)$ are subspaces of 
$H\Lambda^{k}(\Omega)$. Furthermore, each family is nested, i.e., $\mc{Q}_{r-1}^- \Lm^k(\mc{T}_h) \subset \mc{Q}_r^- \Lm^k(\mc{T}_h)$
and  $ \mc{S}_{r-1} \Lm^k(\mc{T}_h) \subset \mc{S}_r \Lm^k(\mc{T}_h)$.
There are also other families of cubical finite element differential forms proposed in the literature, cf. for example \cite{Christiansen-Gillette-2016,Cockburn-Qiu-2014,Gillette-Kloefkorn-2016}, but these spaces will not be used here.

The families $\mc{Q}_r^- \Lm^k$ and $\mc{S}_r \Lm^k$
lead to subcomplexes of the de Rham complex of the form
\algn{ \label{eq:Sr-Qr-exact}
\begin{split}
&\R \mapover{} \mc{Q}_r^- \Lm^0(\mc{T}_h) \mapover{d^0} \mc{Q}_r^- \Lm^1(\mc{T}_h) \mapover{d^1 } \cdots \mapover{d^{n-1}} \mc{Q}_r^- \Lm^n(\mc{T}_h) \mapover{} 0 ,\\
&\R \mapover{} \mc{S}_r \Lm^0(\mc{T}_h) \mapover{d^0} \mc{S}_{r-1} \Lm^1(\mc{T}_h) \mapover{d^1 } \cdots \mapover{d^{n-1}} \mc{S}_{r-n} \Lm^n(\mc{T}_h) \mapover{} 0 .
\end{split}
}
For a given $k$, $0 \le k \le n$, and a given cubical mesh $\mc{T}_h$ the space $\mc{Q}_1^-\Lm^k(\mc{T}_h)$
is the simplest space in the two families above. In fact, we will see below that for any $r \ge 1$ we also have 
\[
\mc{Q}_1^-\Lm^k(\mc{T}_h) \subset \mc{Q}_r^-\Lm^k(\mc{T}_h), \mc{S}_{r} \Lm^k(\mc{T}_h).
\]
Furthermore, in complete analogy with the Whitney forms, $\mc{P}_1^-\Lm^k(\mc{T}_h)$
 in the case of simplicial meshes, the spaces  $\mc{Q}_1^-\Lm^k(\mc{T}_h)$ has a single degree of freedom associated each subsimplex of dimension
 $k$.  More precisely, the degrees of freedom for an element $u \in \mc{Q}_1^-\Lm^k(\mc{T}_h)$ are given by
 \begin{equation}\label{DOF-Q1}
 \int_f \tr_f u , \quad f \in \lap_k(\mc{T}_h).
 \end{equation}
If $T = I_1 \times I_2 \times \ldots \times I_n \in \mc{T}_h$ and
$u \in \mc{Q}_1^-\Lm^k(\mc{T}_h)$, then $u|_T$ is of the form 
\begin{equation}\label{local-Q-}
u|_T = \sum_{\s \in \Sigma(k)} (\sum_{\alpha_j \le 1 - \delta_{j,\sigma}}  c_{\alpha} x^{\alpha}  ) dx_{\s},
\end{equation}
where $\delta_{j,\sigma} =1$ if $j \in \llbracket \s \rrbracket$ and zero otherwise. 
The local spaces $\mc{Q}_1^-\Lm^k(T)$ has dimension $2^{n-k} \pmat{n \\ k}$, and 
together with degrees of freedom \eqref{DOF-Q1} this defines the space $\mc{Q}_1^-\Lm^k(\mc{T}_h)$. In particular, the space 
$\mc{Q}_1^-\Lm^0(\mc{T}_h) = \mc{Q}_1\Lm^0(\mc{T}_h)$,  while 
$\mc{Q}_1^-\Lm^n(\mc{T}_h) = \mc{P}_0\Lm^n(\mc{T}_h)$, i.e., the space of piecewise constant
$n$-forms. Furthermore, for each $k$ with $ 0 < k < n$, the space $\mc{Q}_1^-\Lm^k(\mc{T}_h)$ is strictly contained in $\mc{Q}_1\Lm^k(\mc{T}_h)$.

In \cite{Arnold-Awanou-2014} the definition of the spaces $\mc{S}_r\Lm^k(\mc{T}_h)$ was based on the concept of linear degree. However, a simple and more explicit charaterization of these spaces can be given when $r=1$ .
By utilizing the definition given in \cite{Arnold-Awanou-2014} in this special case we can derive 
that a function $u$ in the space $\mc{S}_1\Lm^k(\mc{T}_h)$ is locally of the form 
\begin{equation}\label{local-S1}
u|_T = u^- + d \kappa \sum_{\s \in \Sigma(k)} \sum_{i \in \llbracket \s \rrbracket} (\sum_{\alpha_j \le 1 - \delta_{j,\sigma}}  c_{\alpha} x^{\alpha}  ) x_i \, dx_{\s}, \quad T \in \mc{T}_h,
\end{equation}
where $u^- \in \mc{Q}_1^-\Lm^k(T)$. 
The local space $\mc{S}_1\Lm^k(T)$ has dimension $2^{n-k} \pmat{n \\ k} (k+1)$, which should be compared with the fact that $\dim \mc{Q}_1\Lm^k(T) = 2^{n} \pmat{n \\ k}$.
Furthermore, the degrees of freedom of the space $\mc{S}_1\Lm^k(\mc{T}_h)$ are given by 
\begin{equation}\label{DOF-S1}
\int_f \tr_f u \wedge v, \qquad v \in \mc{P}_1 \Lm^0 (f), \quad f \in \lap_k(\mc{T}_h).
\end{equation}
In the special cases $k=0$ and $k=n$ we have 
\[
\mc{S}_1\Lm^0(\mc{T}_h) = \mc{Q}_1\Lm^0(\mc{T}_h) \quad \text{and } \mc{S}_1\Lm^n(\mc{T}_h) = \mc{P}_1\Lm^n(\mc{T}_h).
\]
Furthermore, when $n=2$ or $3$
the degrees of freedom of the space
$\mc{S}_1 \Lm^{n-1}(\mc{T}_h)$ corresponds to the degrees of freedom of the ${\rm BDM}_1$ and the ${\rm BDDF}_1$ spaces.  

It follows from \eqref{local-S1} that $d \mc{S}_1\Lm^k(\mc{T}_h) = d \mc{Q}_1^-\Lm^k(\mc{T}_h)$.
Since it is well known that the pair $(\mc{Q}_1^-\Lm^{k-1}(\mc{T}_h), \mc{Q}_1^-\Lm^k(\mc{T}_h))$
is a stable pair for the standard mixed formulation \eqref{hodge-mixed-1-h},
cf. \cite{Arnold-Boffi-Bonizzoni-2015},
it is an easy consequence of this property  that the pair 
$(\mc{S}_1 \Lm^{k-1}(\mc{T}_h),\mc{Q}_1^-\Lm^k(\mc{T}_h))$ also leads to a stable method.
However, as we have already indicated above, the spaces $\mc{S}_1\Lm^{k-1}(\mc{T}_h)$ has to be enriched 
in order to be useful in the present setting, i.e., to give rise to a method with a local coderivative $d_h^*$.
This larger space is introduced below and denoted $\mc{S}_1^+\Lm^k(\mc{T}_h)$.

\subsection{The spaces  $\mc{S}_1^+\Lm^k(\mc{T}_h)$}
For $\mc{S}_1^+ \Lm^k(\mc{T}_h)$ we first define the space of shape functions $\mc{S}_1^+ \Lm^k$. 
We prove that this space is invariant under dilation and translation, then the space of local shape functions on $T$ 
is well-defined as the restriction of $\mc{S}_1^+ \Lm^k$ on $T$. 

To define $\mc{S}_1^+ \Lm^k$, let $\mc{B} \Lm^{k}$ be the span of forms $\{ p_{\s^*}p_{\s} dx_{\s} \, | \, \s \in \Sigma(k) \, \}$,
where $p_{\s} \in \Q_1(\R^k)$ and $p_{\s^*} \in \Q_1(\R^{n-k})$ are polynomials in the variables 
$\{ x_j \}_{j \in \llbracket \s \rrbracket }$
and $\{ x_j \}_{j \in \llbracket \s^* \rrbracket }$, respectively, and where $p_{\s}(0) = 0$.
From this definition of $\mc{B} \Lm^k$ it is obvious that 
\algn{ \label{Q1-direct-sum}
\mc{Q}_1 \Lm^k = \mc{Q}_1^- \Lm^k \oplus \mc{B} \Lm^k .
}
Furthermore, it follows directly from the definition of $\mc{B} \Lm^k$ that 
$d \mc{B} \Lm^k \subset \mc{B} \Lm^{k+1}$. In fact, if $u \in \mc{H}_r \Lm^k \cap \mc{B} \Lm^k$ 
then, by \eqref{eq:homotopy-formula},    $du = (r+k) d \kappa d u \in d\kappa \mc{B} \Lm^{k+1}$. Therefore,  we can conclude that 
\algn{ \label{inclusion-1}
d \mc{B} \Lm^k \subset \mc{B} \Lm^{k+1} \cap d \kappa \mc{B} \Lm^{k+1}.
}
We define $\mc{S}_1^+ \Lm^k$ by 
\algn{ \label{S1p-shape}
\mc{S}_1^+ \Lm^k = \mc{Q}_1^- \Lm^k + d \kappa \mc{B} \Lm^k ,
}
and we now prove that this is invariant under dilation and translation. 
\begin{lemma} \label{invariant}
If $\phi : \R^n \ra \R^n$ is a composition of dilation and translation, then 
$\phi^* \mc{S}_1^+ \Lm^k \subset \mc{S}_1^+ \Lm^k$, where $\phi^*$ is the corresponding pullback.
\end{lemma}
\begin{proof}

Let $\phi (x) = D x + b$ for a given invertible $n \times n$ diagonal matrix $D$ and a vector $b \in \R^n$.
To show $\phi^* \mc{S}_1^+ \Lm^k \subset \mc{S}_1^+ \Lm^k$, assume that $u \in \mc{S}_1^+ \Lm^k$ is written as $u = u^- + d \kappa u^+$ with $u^- \in \mc{Q}_1^- \Lm^k$ and $u^+ \in \mc{B} \Lm^k$. Then we have  
\algns{
\phi^* u = \phi^* u^- + \phi^* d \kappa u^+ = \phi^* u^- + d \phi^* \kappa u^+ = \phi^* u^- + d \kappa \phi^* u^+ + d((\phi^* u^+) \lrcorner b)
}
where we used $\phi^* \kappa u^+ = \kappa \phi^* u^+ + (\phi^* u^+) \lrcorner b$ in the last equality (cf. \cite[Section 3.2]{AFW06}).
We easily check that 
$\phi^* u^- \in \Q_1^-\Lm^k$ and from \eqref{Q1-direct-sum} we have 
\[
d \kappa \phi^* u^+ \in d \kappa \Q_1\Lm^k = d \kappa (\mc{Q}_1^- \Lm^k \oplus \mc{B} \Lm^k) \subset \mc{Q}_1^- \Lm^k 
+ d\kappa \mc{B} \Lm^k = \mc{S}_1^+ \Lm^k.
\]
It remains  to show that 
\algn{ \label{dphi-up}
d ((\phi^* u^+) \lrcorner b) \in \mc{S}_1^+ \Lm^k .
}
To see this, note that $(\phi^* u^+) \lrcorner b \in \mc{Q}_1 \Lm^{k-1}$, 
so $(\phi^* u^+) \lrcorner b  \in \mc{Q}_1^- \Lm^{k-1} \oplus \mc{B} \Lm^{k-1}$ by \eqref{Q1-direct-sum}. 
By \eqref{inclusion-1} we therefore have 
\[
d ((\phi^* u^+) \lrcorner b) \in \mc{Q}_1^- \Lm^{k} + d \kappa \mc{B} \Lm^{k} = \mc{S}_1^+ \Lm^k,
\]
so \eqref{dphi-up} is established.
%
%
\end{proof}
We define the space of shape functions of $\mc{S}_1^+ \Lm^k(T)$ on an element 
$T \in \mc{T}_h$ as the restriction of the functions 
in the class $\S_1^+ \Lm^k$ such that 
\algn{ \label{local-S1+new}
\S_1^+ \Lm^k (T) = \Q_1^- \Lm^k (T) + d \kappa \mc{B} \Lm^k(T) .
}
By comparing the definition above with the characterizations of the spaces $\mc{Q}_1^-\Lm^k(T)$ and $\mc{S}_1\Lm^k(T)$
we can conclude that  $\mc{S}_1^+\Lm^k(T)$ contains these spaces. 
It also follows directly from the definition that $\dim \mc{B} \Lm^k(T) = \pmat{n \\ k} 2^{n-k} (2^k - 1)$, and therefore
we must have  
\begin{equation}\label{dim-leq}
\dim  \mc{S}_1^+\Lm^k(T) \le \dim \mc{Q}_1^-\Lm^k(T) + \dim \mc{B} \Lm^k(T)
=2^{n} \pmat{n \\ k}.
\end{equation}
In fact, we will show below that this inequality is an equality, and that the degrees of freedom for this space is 
\begin{equation}\label{DOF-S1+}
\int_f \tr_f u \wedge v, \qquad v \in \mc{Q}_1 \Lm^0 (f), \quad f \in \lap_k(T).
\end{equation}
Furthermore, it will follow from the discussion below that for any $u \in \mc{S}_1^+\Lm^k(\mc{T}_h)$ the degrees of freedom associated an interface 
$f \in \lap_{n-1}(\mc{T}_h)$ determines $\tr_f u$ uniquely. As a consequence, 
$\mc{S}_1^+\Lm^k(\mc{T}_h) \subset H\Lm^k(\Omega)$, and 
\[
\mc{Q}_1^- \Lm^k(\mc{T}_h) \subset \mc{S}_1\Lm^k(\mc{T}_h) \subset \mc{S}_1^+\Lm^k(\mc{T}_h).
\]
From the definition above we can easily derive that $\mc{S}_1^+\Lm^0(\mc{T}_h) = \mc{Q}_1\Lm^0(\mc{T}_h)= \mc{Q}_1^-\Lm^0(\mc{T}_h)$,
and for $k=n$ it is a consequence of  \eqref{eq:homotopy-formula} that 
$d \kappa \mc{B} \Lm^n (T) = \mc{B} \Lm^n(T)$. Therefore,  $\S_1^+ \Lm^n(\mc{T}_h) = \Q_1 \Lm^n (\mc{T}_h)$.  
Moreover, from \eqref{local-S1+new} we have that $ d \S_1^+ \Lm^k(\mc{T}_h) =     d\Q_1^- \Lm^k(\mc{T}_h) $.
As above, we can therefore conclude that the pair $(\mc{S}_1^+\Lm^{k-1}(\mc{T}_h), \Q_1^- \Lm^k(\mc{T}_h) )$
is a stable pair for the mixed formulation \eqref{hodge-mixed-1-h}. Moreover, in the present case we will be able to 
construct a suitable integration rule such that conditions {\bf (A)} and {\bf (B)} of Section \ref{sec-abstract} are fulfilled,
and which leads to a local coderivative $d_h^*$. However, first we need to analyze the spaces
$\S_1^+ \Lm^k(\mc{T}_h)$ just introduced.

If $m$ is a $k$-form given by $m = p dx_{\s}$, where $\s \in \Sigma(k)$ and the coefficient polynomial $p(x)$ is a monomial,
then we will refer to $m$ as a form monomial.

\begin{lemma} \label{lemma:Blm} The following hold:
\begin{itemize}
 \item[(a)] \label{B1m-1} For a form monomial $m\not = 0$ in $\mc{B} \Lm^k(T)$, $\kappa m$ generates at least one form monomial such that its coefficient contains a quadratic factor. 
 \item[(b)] For $u \in \mc{B} \Lm^k(T)$, the coefficient of each form monomial of $d \kappa u$ has at most one quadratic factor. 
 \item[(c)] For $u \in \mc{B} \Lm^k(T)$ and $f \in \lap_k(T)$, $\tr_f (d \kappa u) \in \Q_1 \Lm^k(f)$. 
 \item[(d)] The operator $d\kappa$ is injective on $\mc{B} \Lm^k(T)$.
\end{itemize}
\end{lemma}
\begin{proof}
Let us define $\mc{B}$ as 
\algn{ \label{eq:BLm-basis}
\mc{B} = \{ p_{\s^*} p_{\s} dx_\s \in \mc{B}\Lm^k(T) \;|\; \s \in \Sigma(k) \},
}
where $p_{\s}(x) = x^{\alpha}$ and $p_{\s^*}(x) = x^{\beta}$ are monomials in $\mc{Q}_1$ of the variables $\{ x_j \}_{j \in \llbracket \s \rrbracket }$
and $\{ x_j \}_{j \in \llbracket \s^* \rrbracket }$, respectively, and where $|\alpha| \ge 1$. The set $\mc{B}$ is a basis for $\mc{B} \Lm^k(T)$.

For $p_{\s^*} p_{\s} dx_{\s} \in \mc{B}$, $\kappa (p_{\s^*} p_{\s} dx_{\s})$ is a linear combination of
\algn{  \label{m_i}
m_i = (-1)^{i+1} p_{\s^*} p_{\s} x_{\s_i} dx_{\s_1} \wedge \cdots \wedge \widehat{dx_{\s_i}} \wedge \cdots \wedge dx_{\s_k} , \qquad 1 \leq i \leq k .
}
Since $p_\s$ has a factor $x_{\s_j}$ for some $\s_j \in \jump{\s}$, the coefficient of $m_j$ has  $x_{\s_j}^2$ as factor,
so claim (a) is proved. Furthermore, 
a direct computation gives 
\algns{
d m_i &= \pd_{\s_i} (p_{\s^*} p_{\s} x_{\s_i}) dx_{\s} \\
& \quad + (-1)^{i+1} \sum_{j \in \jump{\s^*}} (\pd_j p_{\s^*}) p_\s x_{\s_i} dx_{j} \wedge dx_{\s_1} \wedge \cdots \wedge \widehat{dx_{\s_i}} \wedge \cdots \wedge dx_{\s_k} .
}
and each of these coefficients has at most one quadratic factor. Therefore, claim (b) is proved.

To prove (c), it is enough to show that $\tr_f d m_i \in \Q_1 \Lm^k (f)$ for any $f \in \lap_k(T)$ and $1 \le i \le k$. 
Recall that $f \in \lap_k(T)$ is determined by fixing values of $n-k$ variables. Let $I(f) \subset \{1, \ldots, n\}$ be the set of indices such that $f$ is determined by fixing $x_l$ for all $l \in I(f)$.  By letting $\vol_f$ be the volume form on $f$, we have,
up to a sign, that 
\algns{
\tr_f (d m_i) = 
\case{ 
(\pd_{\s_i} (p_{\s^*} p_\s x_{\s_i}))|_f \vol_f , & \text{ if } I(f) = \jump{\s^*}, \\
({x}_{\s_i} (\pd_j p_{\s^*}) p_\s )|_f \vol_f, &\text{ if } I(f) = \{ \s_i \} \cup \jump{\s^*} \setminus \{j\}, \\
0, & \text{ otherwise.} 
}
}
Since all variables in $p_{\s^*} p_\s$ have degree at most 1, the same is the case for $\pd_{\s_i} (p_{\s^*} p_\s x_{\s_i})$,
while all variables in ${x}_{\s_i} (\pd_j p_{\s^*}) p_\s$ have degree at most 1, with a possible exception  for $x_{\s_i}$
 which is constant on $f$.
So claim (c) follows.

To prove the injectivity of $d\kappa$ on $\mc{B} \Lm^k(T)$, we first prove that $\kappa$ is injective.
To see this consider  two distinct elements 
$m= p_{\s^*} p_\s dx_\s$ and $\tilde m = p_{\tilde{\s}^*} p_{\tilde{\s}} dx_{\tilde{\s}}$ of  $\mc{B}$. 
We claim that all the monomials generated by $\kappa m$ and $\kappa \tilde m$
which have a quadratic factor, 
are also distinct. To see this we assume the contrary, i.e., that there are $i \in \jump{\sigma}$ and 
$\tilde i \in \jump{\tilde \sigma}$ such that $\jump{\sigma} \setminus \{i\} = \jump{\tilde \sigma} \setminus \{\tilde i\}$
and 
\begin{equation}\label{identity-quad}
x_i p_{\s^*} p_\s = \pm x_{\tilde i}p_{\tilde{\s}^*} p_{\tilde{\s}},
\end{equation}
where the left-hand side is quadratic in $x_i$ and the right-hand side is quadratic in $x_{\tilde i}$.  
Since $p_{\s^*} p_\s, p_{\tilde{\s}^*} p_{\tilde{\s}} \in \Q_1$, this can be true only when $i = \tilde i$, i.e., 
$\sigma = \tilde \sigma$. However, \eqref{identity-quad} implies that $m = \tilde m$, which is a contradiction. 
This implies that the elements of $\kappa(\mc{B})$ are linearly independent, and therefore $\kappa$ is injective on 
 $\mc{B} \Lm^k(T)$.
Finally, since  $d$ is injective on the image of $\kappa$ by \eqref{eq:homotopy-formula}, 
$d \kappa$ is injective on $\mc{B} \Lm^k(T)$.
\end{proof}

The following key result is a consequence of the lemma just established.
\begin{theorem} For $T \in \mc{T}_h$ and $ 0 \leq k \leq n$ we have 
\algns{
\dim \mc{S}_1^+ \Lm^k  (T)= 2^n \pmat{n \\ k} .
}
\end{theorem}
\begin{proof}
By Lemma~\ref{lemma:Blm}~(d), the spaces $d\kappa \mc{B} \Lm^k(T)$ and  $\mc{B} \Lm^k(T)$ have the same dimension, therefore
the conclusion will follow if we show that the sum \eqref{local-S1+new} is a direct sum, cf. \eqref{dim-leq}.
%
To show that the sum \eqref{local-S1+new} is direct, it is enough to show 
\algn{ \label{D-intersect}
D \cap d \kappa \mc{B} \Lm^k(T) = \{0\} , \qquad D:= \ker d \cap \Q_1^- \Lm^k(T) .
}
Note that $D \cap \mc{B} \Lm^k(T) = \{0\}$ due to \eqref{Q1-direct-sum}. 
Furthermore, by
\eqref{eq:homotopy-formula}, $D = (d \kappa + \kappa d) D = d \kappa D$. 
Therefore, \eqref{D-intersect} will follow if we can show that $d\kappa$ is injective on 
$D \oplus \mc{B} \Lm^k(T)$. 
However, by \eqref{eq:homotopy-formula} this will follow if we can show that  $\kappa$ 
is injective on $D \oplus \mc{B} \Lm^k(T)$. To see why this is the case we observe that 
$\kappa(\Q_1^- \Lm^k(T)) \subset \Q_1^- \Lm^{k-1}(T)$.  Combined with 
Lemma~\ref{lemma:Blm}~(a) this implies that $\kappa D \cap \kappa \mc{B} \Lm^k(T) = \{0\}$.
As a consequence, $\kappa$ is injective on $D \oplus \mc{B} \Lm^k(T)$ if it is injective 
on each of the spaces $D$ and $\mc{B} \Lm^k(T)$ separately. The latter statement  follows from 
Lemma~\ref{lemma:Blm}~(d), while the injectivity on $D$ follows from the fact 
that $\ker d \cap \ker \kappa = \{0\}$ by \eqref{eq:homotopy-formula}.
\end{proof}
We are now ready to prove unisolvency of $\mc{S}_1^+ \Lm^k(T)$ with the degrees of freedom \eqref{DOF-S1+}. 
\begin{theorem}\label{thm:unisolvency}
An element $u \in \mc{S}_1^+ \Lm^k(T)$ is uniquely determined by the degrees of freedom \eqref{DOF-S1+}.  
\end{theorem}
\begin{proof} 
Since $\dim \mc{Q}_1 (f) = 2^k$ for $f \in \lap_k(T)$, the number of degrees of freedom  in \eqref{DOF-S1+} is 
\algns{
|\lap_k (T)|  \times 2^k = \pmat{n \\ k} 2^{n-k} \times 2^k = 2^n \pmat{n \\ k},   
}
which is same as $\dim  \mc{S}_1^+ \Lm^k(T)$. 
Therefore, it is enough to show that if 
$u \in \mc{S}_1^+ \Lm^k(T)$, and all the degrees of freedom \eqref{DOF-S1+} vanish, then $u= 0$.

By dilation and translation, we may assume that $T$ is the unit hypercube $[0,1]^n \subset \R^n$, cf. Lemma~\ref{invariant}.
Suppose that $u = \sum_{\s \in \Sigma(k)} u_\s dx_\s \in \S_1^+ \Lm^k(T)$ and all the degrees of freedom 
\eqref{DOF-S1+} of $u$ are zero. From Lemma~\ref{lemma:Blm} (c) we can conclude that 
$\tr_f u = 0$ for all $f \in \Delta_k(T)$. In particular, the coefficient $u_\s$ vanish for all faces $f$ where $x_i$ is fixed in $\{0,1\}$
for $i \in \jump{\s^*}$, and as a consequence $u_\s$
has 
$ \prod_{i \in \jump{\s^*}} x_i(1- x_i)$ as a factor. Therefore, if  $k < n-1$, it follows from 
Lemma~\ref{lemma:Blm} (b) that $u=0$.
Furthermore, if $k= n$, then $\S_1^+ \Lm^n(T) = \Q_1 \Lm^n (T)$, so $u=0$ is a direct consequence of the degrees of freedom in this case.

It remains to cover the case 
$k = n-1$. In this case $u$ can be written as 
\algns{
u = \sum_{i=1}^n u_i dx_{\s(i)},  \quad \text{where }dx_{\s(i)} := dx_1 \wedge \cdots \wedge \widehat{dx_i} \wedge \cdots \wedge dx_n. 
}
From the discussion above we already know that the coefficients
$u_i$ have $x_i(1-x_i)$ as factors. In other words, 
\begin{equation}\label{u_i-rep}
u_i = c_i x_i (1 - x_i), \quad i =1,2, \ldots ,n,
\end{equation}
where the coefficients $c_i$ may depend on $x$, but they are independent of $x_i$.
Furthermore, since $\tr u = 0$ on the boundary of $T$ and $du$ is a constant 
$n$-form, we can conclude from Stokes theorem that $du \equiv 0$.
Furthermore, $u$ is of the form
\[
u = \sum_{i=1}^n (a_i + b_i x_i ) dx_{\s(i)} + d \kappa u^+,
\]
with constant coefficients  $a_i$, $b_i$, and $u^+ \in \mc{B} \Lm^{n-1}(T)$. From the definition of the space $ \mc{B} \Lm^{n-1}(T)$ we have that
\[
u^+= \sum_{i=1}^n u^+_i dx_{\s(i)}, \qquad \text{with }   u^+_i = p_i + x_i q_i, 
\]
where the polynomials $p_i$ and $q_i$ are in $\Q_1$, independent of the variable $x_i$, and satisfies 
$p_i(0) = q_i(0) = 0$.

Note that 
\[
\kappa u^+_i dx_{\s(i)} = \sum_{j =1}^{i-1}(-1)^{j+1} x_j u^+_i dx_{\s(j,i)}
- \sum_{j =i+1}^{n}(-1)^{j+1} x_ju^+_i dx_{\s(i,j)},
\]
where $dx_{\s(i,j)}$, for $i <j$, is the $n-2$ form obtained from $dx_1 \wedge dx_2 \wedge \cdots \wedge dx_n$ by omitting $dx_i$ and $dx_j$.
If we let $v = d\kappa u^+$ then a further calculation using the definition of the exterior derivative gives
\[
v = \sum_{i=1}^n v_i dx_{\s(i)}, \quad \text{where } 
v_i = \sum_{j=1}^n \partial_j(x_j u^+_i + (-1)^{i+j+1}x_i u^+_j).
\]
From this it follows that the coefficients $u_i$ of $u$ can be represented as 
$u_i = u_i^1 + u_i^2$ where 
\[
u_i^1 = a_i + \sum_{j=1}^n \partial_j(x_j p_i + (-1)^{i+j+1} x_i p_j)
\quad \text{and } 
u_i^2 = x_i [ b_i + \sum_{j=1}^n \partial_j(x_j q_i + (-1)^{i+j+1} x_j q_j)].
\]
We observe that all terms in this expression for $u_i$, except for $a_i + \sum_{j \neq i} \partial_j(x_j p_i) $, has $x_i$ as a factor. In fact, this term is independent of the variable $x_i$, and therefore we must have 
\[
a_i + \sum_{j \neq i} \partial_j(x_j p_i) = a_i + (n-1)p_i + \sum_{j =1}^n x_j \partial_jp_i 
\equiv  0.
\]
However, by using \eqref{eq:homotopy-formula} in the special case of zero forms, we easily see that the only 
possible solution is $p_i = - a_i/(n-1)$. In particular, since $p_i(0) = 0$, we can conclude that both $p_i$ and $a_i$ are zero.
Therefore, $u_i = u_i^2 = x_i \tilde u_i$, where 
\begin{equation}\label{tilde-u-rep}
\tilde u_i = b_i + \sum_{j=1}^n \partial_j(x_j q_i + (-1)^{i+j+1} x_j q_j), \quad i =1,2, \ldots ,n.
\end{equation}
As a consequence, we obtain 
\[
\partial_i \tilde u_i = \sum_{j=1}^n \partial_i \partial_j(x_j q_i + (-1)^{i+j+1} x_j q_j) = \sum_{j=1}^n (-1)^{i+j+1}\partial_i q_j.
\]
However, by \eqref{u_i-rep}
we also have  $\tilde u_i = c_i(1-x_i)$ and $\partial_i \tilde u_i  = - c_i$,
and therefore we obtain
\[
c_i =  (-1)^{i} \partial_i Q, \quad \text{where } Q = \sum_{j=1}^n (-1)^j q_j.
\]
The equation we obtain from the two representations of $\tilde u_i$ can be written
\[
(-1)^{i} [b_i + nq_i + \sum_{j=1}^n x_j \partial_j q_i ] - Q = (1- x_i)\partial_i Q.
\]
Note that $du = \sum_i (-1)^{i+1} b_i = 0$ and therefore, by summing the equation above over $i$, we obtain 
\[
\sum_{i=1}^n (x_i - \frac{1}{2}) \partial_i Q = 0.
\]
However $Q \in \Q_1$, and by expanding $Q$ in monomials with respect to the variables $x_i - \frac{1}{2}$ we can conclude that 
$Q$ is a constant. Furthermore, it vanishes at the origin, so $Q \equiv 0$. From \eqref{tilde-u-rep} we then obtain that
\[
\frac{u_i}{x_i} = \tilde u_i = b_i + \sum_{j=1}^n \partial_j (x_j q_i),
\]
which is independent of the variable $x_i$. By \eqref{u_i-rep} this implies that each $u_i$ is zero.
This completes the proof.
\end{proof}

Next we consider the traces of elements in $\mc{S}_1^+ \Lm^k(T)$ on $f \in \lap_{n-1}(T)$. Since $f$ is defined by fixing 
one coordinate, the other $n-1$ variables define a coordinate system on $f$. In particular, we can define the corresponding Koszul operator $\kappa_f$ for differential forms on $f$, 
cf. \eqref{kos-tr}, and as a consequence the space $\mc{S}_1^+ \Lm^k(f)$ is defined by 
the embedding of $f$ into $\R^{n-1}$.

\begin{theorem} \label{thm:tr-inclusion}
If  $f \in \lap_{n-1}(T)$ and $k \le n-1$, then  
\algn{ \label{tr-inclusion0}
\tr_f \mc{S}_1^+ \Lm^k(T) \subset \mc{S}_1^+ \Lm^k(f). 
}
\end{theorem}
\begin{proof}
Since the trace operator maps $\mc{Q}_1^- \Lm^k (T)$ into $\mc{Q}_1^- \Lm^k(f)$,
we only have to show that $\tr_f ( d \kappa \mc{B}\Lm^k(T)) \subset \mc{S}_1^+ \Lm^k(f)$.
%
%

Without loss of generality, we may assume that $f = \{x \in \R^n \,:\, x_1 = c\}$ for  a constant $c$. 
Note that the definition of $\mc{B} \Lm^k(T)$ then implies that 
\begin{equation}\label{tr-B}
 \tr_f \mc{B} \Lm^k(T) \subset \mc{B} \Lm^k(f).
 \end{equation}
Furthermore, any nonzero form monomial $u = u_\s dx_\s \in \mc{B} \Lm^k(T)$ satisfies one of the following
conditions:
\begin{itemize}
\item[i)] $1 \not \in \jump{\s}$,
\item[ii)] $1 \in \jump{\s} $ and there exists $i \in \jump{\s},\, i \not = 1, $ such that $u_\s$ has $x_i$ as a factor,
\item[iii)] $1 \in \jump{\s} $ and there exists no $i \in \jump{\s},\, i \not = 1, $ such that $u_\s$ has $x_i$ as a factor.
\end{itemize}
%
%

We will  prove that $\tr_f d \kappa u \in \mc{S}_1^+ \Lm^k(f)$ in each of these cases.
In case i) it follows from \eqref{kos-tr} that $\tr_f \kappa u = \kappa_f \tr_f u$, since $u \lrcorner x^f = 0$. Therefore, $\tr_f d \kappa u = d \kappa_f \tr_f u $, and this is in $d\kappa_f \mc{B} \Lm^k(f) \subset \mc{S}_1^+ \Lm^k(f)$ by \eqref{tr-B}.
In case ii) and iii) we write $u$ as $u_\s dx_1 \wedge dx_{\eta}$, where $\jump{\eta} \subset \{2, \ldots ,n\}$.
A direct computation shows that $\tr_f \kappa u = c u_\s dx_{\eta}$.
In case ii)
there is $i \in \jump{\eta}$ such that 
$u_\s$ has $x_i$ as a factor. Therefore, $\tr_f \kappa u \in \mc{B} \Lm^{k-1}(f)$, and  
\eqref{inclusion-1} implies that $\tr_f d \kappa u  = d \tr_f \kappa u$ is in 
$d \kappa_f \mc{B} \Lm^{k}(f) \subset \mc{S}_1^+ \Lm^k(f)$. 
Finally, in case iii) $\tr_f \kappa u \in \Q_1^-\Lm^{k-1}(f)$, and therefore 
\[
\tr_f d\kappa u = d \tr_f \kappa u \in \Q_1^-\Lm^{k}(f) \subset \mc{S}_1^+ \Lm^k(f).
\]
This completes the proof.
\end{proof}
The inclusion \eqref{tr-inclusion0} is indeed an equality. In fact, this follows since an element of $\mc{S}_1^+ \Lm^k(f)$
is uniquely determined by degrees of freedom associated the elements of $\lap_k(f)$.
Furthermore, the trace result can be used repeatedly to conclude that 
\[
\tr_f \mc{S}_1^+ \Lm^k(T) = \mc{S}_1^+ \Lm^k(f), \quad f \in \lap(T), \: n \ge \dim f \ge k.
\]
In particular, if $\dim f = k$ we have 
\begin{equation}\label{k-trace}
\tr_f \mc{S}_1^+ \Lm^k(T) = \mc{Q}_1\Lm^k(f), \quad f \in \lap_k(T).
\end{equation}
An important consequence of the combination of the Theorems \ref{thm:unisolvency} and \ref{thm:tr-inclusion}
is also that the space $\mc{S}_1^+ \Lm^k(\mc{T}_h)$ is a subspace of $H\Lm^k(\Omega)$, since the traces are continuous over 
elements of $\lap_{n-1}(\mc{T}_h)$.
Furthermore, as we have already indicated above, it is a consequence of the fact that 
$d\Q_1^- \Lm^k(\mc{T}_h) = d \S_1^+ \Lm^k(\mc{T}_h)$ and the stability of the method derived from the $\mc{Q}_1^-\Lm^k$ spaces,
 that the pair $(\mc{S}_1^+\Lm^{k-1}(\mc{T}_h), \Q_1^- \Lm^k(\mc{T}_h) )$
is a stable pair for the mixed formulation \eqref{hodge-mixed-1-h}. 
Therefore, according to the abstract theory in Section \ref{sec-abstract}, to obtain a convergent method with a local 
coderivative  $d_h^*$, we need to define a proper  integration rule such that conditions  {\bf (A)} and {\bf (B)} holds.

\subsection{The local method}\label{local-c}
It is a consequence of the standard error estimate \eqref{eq:tilde-approx} that the choices $V_h^{k-1} = \mc{S}_1^+ \Lm^{k-1}(\mc{T}_h)$ and $V_h^k = \Q_1^- \Lm^k(\mc{T}_h)$ for 
the standard mixed method \eqref{hodge-mixed-1-h} will, under the assumption of a sufficiently regular solution,
lead to an  estimate for the error in the energy norm of order $O(h)$. Therefore, the goal is to perturb the method such that we preserve this convergence order, and also local coderivatives $d_h^*$. As in the simplicial case the discussion is based on the abstract theory of 
Section~\ref{sec-abstract}. Furthermore, $W^k = L^2\Lm^k(\Omega)$, $ V^k = H\Lm^k(\Omega)$,
and $\LRa{ \cdot, \cdot}$ is used to denote appropriate $L^2$ inner products.

In the present case condition {\bf (B)}
will appear slightly more complicated than in the simplicial case, since the space $\tilde{V}_h^{k-1}$ is strictly contained 
in $V_h^{k-1}$. In fact, we will take $\tilde{V}_h^{k-1} = \mc{Q}_1^- \Lm^{k-1}(\mc{T}_h)$ and as in the simplicial case 
the space 
$W_h^{k-1}$ is given by \eqref{eq:Wh-simplicial}, i.e., it consists of piecewise constant $(k-1)$-forms.
As a consquence, it follows from Theorem~\ref{thm:main} that if we are able to define a modified inner product on 
$\mc{S}_1^+ \Lm^{k-1}(\mc{T}_h)$ such that  conditions  {\bf (A)} and {\bf (B)} hold with these choices, then the linear convergence is obtained. 

As in the simplical case our choice of modified  inner product can be motivated from an alternative set of degrees of freedom 
for the spaces $\mc{S}_1^+ \Lm^{k}(\mc{T}_h)$. The degrees of freedom for this space given by \eqref{DOF-S1+} shows that 
the global dimension of this space is given by 
\[
\dim \mc{S}_1^+ \Lm^{k}(\mc{T}_h) = 2^k | \lap_k(\mc{T}_h) |.
\]
In particular, $\tr_f u$ for $f \in \lap_k(\mc{T}_h)$ and $ u \in \mc{S}_1^+ \Lm^{k}(\mc{T}_h)$ is uniquely determined 
by the $2^k$ degrees of freedom associated $f$. However, elements of $\Q_1 \Lm^k(f)$ can be identified by an element in 
$\Q_1 \Lm^0(f)$, and therefore $\tr_f u$ is also determined by the values of $\tr_f u$ at the $2^k$ vertices  of $f$.
More precisely, for each $f \in \lap_k(\mc{T}_h)$ and each $x_0 \in \lap_0(f)$ we define the functional $\phi_{f,x}$ by 
\[
\phi_{f,x_0}(u) = u_x(x_1 - x_0, x_2 - x_0, \ldots ,x_k - x_0),
\]
where $\{ x_j \}_{j=1}^k$ are the $k$ vertices of $f$ such that $[x_0,x_j] \in \lap_1(f)$. The functionals $\phi_{f,x}$ for $x \in \lap_0(f)$
will determine $\tr_f u$ uniquely, and the set $\{\phi_{f,x} \: | \: f \in \lap_k(\mc{T}_h), \: x \in \lap_0(f) \: \}$ will be a set of global degrees of freedom of $\mc{S}_1^+ \Lm^{k}(\mc{T}_h)$. Furthermore, if $T \in \mc{T}_h$, then restriction of 
$u \in \mc{S}_1^+ \Lm^{k}(\mc{T}_h)$ to $T$ at the vertex $x_0 \in \lap_0(T)$ is determined by the 
 $\pmat{n \\ k}$ possible choices of $\phi_{f,x_0}(u) $ for $f \in \lap_k(T)$ such that $x_0 \in f$.

We will let $\{ \psi_{f,x} \}$ be the corresponding dual basis for the space $\mc{S}_1^+ \Lm^{k}(\mc{T}_h)$,
defined by 
\[
\phi_{g,y} (\psi_{f,x}) = \delta_{(f,x),(g,y)}, \qquad f,g \in \Delta_k(\mc{T}_h), \, x \in \Delta_0(f), \, y \in \Delta_0(g).
\]
The modified inner product $\LRa{\cdot, \cdot }_h$ on $\mc{S}_1^+ \Lm^{k}(\mc{T}_h)$ is now defined by
\begin{equation}\label{quadrature-c}
\LRa{u, v }_h = \sum_{T \in \mc{T}_h} \LRa{u, v }_{h,T}, \quad \text{where } 
\LRa{u, v}_{h,T} = 2^{-n} |T|  \sum_{x \in \lap_0(T)} \LRa{u_{x}, v_{x} }_{\Alt}.
\end{equation}
It follows from the discussion of degrees of freedom above that
the quadratic form $\LRa{\cdot, \cdot }_h$  is an inner product $\mc{S}_1^+ \Lm^{k}(\mc{T}_h)$,
and a standard scaling argument shows that it is equivalent to the standard $L^2$ inner product.
So condition {\bf (A)} holds.

Next, we will verify condition {\bf (B)}, but with $k-1$ replaced by $k$ to simplify the notation. 
We observe that the inner product $\LRa{\cdot, \cdot }_h$ satisfies
\[
\LRa{u, v }_h = \LRa{u, v }, \quad u \in \Q_1\Lm^k(\mc{T}_h), \: v \in W_h^k.
\]
As a consequence, \eqref{eq:B-cond} holds for $\tilde{V}_h^k = \Q_1^-\Lm^k(\mc{T}_h) \subset \Q_1\Lm^k(\mc{T}_h)$.
However, note that in general the space $\mc{S}_1^+ \Lm^{k}(\mc{T}_h)$ will contain quadratic terms and therefore 
the identity \eqref{eq:B-cond} will not hold if  $\tilde{V}_h^k$ is replaced by  $V_h^k$.
To complete the verification of condition {\bf (B) } we have to define a projection $\Pi_h : \mc{S}_1^+ \Lm^{k}(\mc{T}_h) \to 
 \Q_1^-\Lm^k(\mc{T}_h)$ which satisfies $d \Pi_h u = du$. We define this projection by the degrees of freedom 
\eqref{DOF-Q1}, i.e.,
\[
\int_f \tr_f \Pi_h u  = \int_f \tr_f  u  , \quad f \in \lap_k(\mc{T}_h).
\]
Note that it follows from the definition of the spaces $\mc{S}_1^+ \Lm^{k}(\mc{T}_h)$ and $\Q_1^-\Lm^k(\mc{T}_h) $
that both $d\Pi_h u$ and $du$ are piecewise constant forms, and by Stokes' theorem they are equal.
Furthermore, since the degrees of freedom of $\Q_1^-\Lm^k(\mc{T}_h) $ is a subset of the degrees of freedom of
$\mc{S}_1^+ \Lm^{k}(\mc{T}_h)$, the uniform $L^2$ boundedness of $\Pi_h$ is a consequence of equivalence 
of the $L^2$ norm and a discrete norm defined by the degrees of freedom on each of these spaces.
Finally, it remains to verify \eqref{eq:B-cond2}, i.e., we need verify that 
\begin{equation}\label{eq:B-cond2b}
\LRa{\Pi_h u, v }_h = \LRa{u, v }_h, \quad u \in \mc{S}_1^+ \Lm^{k}(\mc{T}_h), v \in W_h^k.
\end{equation}
To see this we observe that 
\algns{
\LRa{u, v }_{h,T} &= 2^{-n} |T| \sum_{x \in \lap_0(T)} \LRa{u_{x}, v_{x} }_{\Alt} \\
&= 2^{-n} |T| \sum_{f \in \lap_k(T)} \sum_{x \in \lap_0(f)} \LRa{(\tr_f u)_x ,(\tr_f v)_x}_{\Alt(f)},
}
where the subscript $\Alt(f)$ indicates the inner product of alternating $k$-forms on $f$. Furthermore, since 
$\tr_f u \in \Q_1\Lm^k(f)$ by \eqref{k-trace}
and $\tr_f v \in \mc{P}_0 \Lm^k(f)$, we have 
\[
\begin{split}
2^{-k} \sum_{x \in \lap_0(f)} \LRa{(\tr_f u)_x ,(\tr_f v)_x}_{\Alt(f)} &= | f |^{-1} \int_f \LRa{ \tr_f u, \tr_f v}_{\Alt(f)} \: \vol_f \\
&= | f |^{-1} \int_f \LRa{ \tr_f \Pi_h u, \tr_f v}_{\Alt(f)} \: \vol_f \\
&= 2^{-k} \sum_{x \in \lap_0(f)} \LRa{(\tr_f \Pi_h u)_x ,(\tr_f v)_x}_{\Alt(f)},
\end{split}
\]
and hence the desired identity \eqref{eq:B-cond2b} holds. We have therefore verified
condition {\bf (B) }.

Finally, we need to convince ourselves that the corresponding operator $d_h^*$, defined by 
\[
\LRa{ d_h^* u, \tau}_h = \LRa{  u, d\tau }, \quad u \in \Q_1^-\Lm^k(\mc{T}_h), \tau 
\in \mc{S}_1^+ \Lm^{k}(\mc{T}_h),
\]
is local. However, since the mass matrix $\LRa{\psi_{f,x}, \psi_{g,y} }_h$ is block diagonal, where the blocks correspond to 
the vertices of $\mc{T}_h$, we can argue exactly as we did 
 in the proof of Theorem~\ref{local-simplicial} above to establish this property.

\section{Concluding remarks} \label{conclusion}
We have carried out the construction of finite element methods for the Hodge Laplace problems that admit local approximations of the coderivatives. Constructions are performed both with respect to simplicial and cubical meshes. These methods will therefore correspond to methods where the approximation of local constitutive laws are local, in contrast to the properties of more standard mixed finite 
element methods. The methods are of low order, and can also be seen as finite difference methods. However, 
an advantage of our approach is that there is a natural path to convergence estimates, based on 
standard finite element theory and variational crimes.

In the study above we have only considered Hodge-Laplace problems with constant  coefficients. However, 
as we have stated in Section~\ref{prelim}  above, the discussion above can easily be adapted to variable coefficients.  
To be more precise, the challenge for such a construction is to allow variable coefficients in the  
$dd^*$ term of the operator $L$, cf. \eqref{hodge-Lap-strong}. Consider the case when this term is modified to a term of the form $d \bs{K} d^*$.  We then end up in the situation where 
the corresponding weighted inner products 
\[
\LRa{ u, v}_{\bs{K}} := \int_{\Omega} \LRa{ \bs{K}^{-1}u_x, v_x}_{\Alt} \, dx 
\]
have to be replaced by a quadrature rule, $\LRa{ u, v}_{\bs{K},h}$, leading to a block diagonal ``mass matrix'',
and satisfying conditions {\bf (A)} and {\bf (B)}. 
Here we assume that $\bs{K}$ is piecewise constant with respect to the 
mesh $\mathcal{T}_h$, and that for each $T \in \mathcal{T}_h$ the operator $\bs{K}_T : \Alt^{k} \to \Alt^{k}$ is symmetric and positive definite 
with respect to the inner product $\LRa{ \cdot, \cdot }_{\Alt}$, and with uniform upper and lower bounds on its eigenvalues. 
In the present setting we simply change the term
\[
\sum_{x \in \Delta_0(T)} \LRa{ u_x, v_x }_{\Alt} \quad \text{to} \quad \sum_{x \in \Delta_0(T)} \LRa{\bs{K}_{x,T}^{-1} u_x, v_x }_{\Alt}
\]
in the quadrature rules \eqref{quadrature-s} and \eqref{quadrature-c}.
Here $\bs{K}_{x,T}$ denotes the value of $\bs{K}$ at the vertex $x$, and taken from the element $T \in \mathcal{T}_h$.
With these definitions it is straightforward to 
verify condition {\bf (A)}, and condition {\bf (B)} in the simplicial case. 
The same holds for the identity \eqref{eq:B-cond} in the cubical case, where the spaces $\tilde{V}_h^k$ and $W_h^k$  are defined as in Section~\ref{local-c}. Furthermore, 
we can 
utilize the symmetry 
of $\bs{K}$ to verify \eqref{eq:B-cond2b} that 
\[
\LRa{\Pi_h  u,v }_{\bs{K},h } = \LRa{  u,v }_{\bs{K}, h }, \quad u \in \mc{S}_1^+ \Lm^{k}(\mc{T}_h), v \in W_h^k . 
\]
Here the operator $\Pi_h$ is defined exactly as in Section~\ref{local-c} above.

There are also other possible quadrature rules than the ones we have used  above that we could have considered in the present study. In particular, 
in the cubical case the combination of Gauss and Lobatto points, as suggested in \cite{MonkCohen1998},
may even lead to higher order methods. However, this approach is based on the use of different quadrature rule
in each component, and therefore it will lead to methods which are not robust for variable coefficients, in contrast to the above. 
Since the operator $\bs{K}$ may  mix the components, this will affect the accuracy of the quadrature rule,
and the lowest order methods may not even converge. Therefore, we have not 
considered this approach in this paper.

\providecommand{\bysame}{\leavevmode\hbox to3em{\hrulefill}\thinspace}
\providecommand{\MR}{\relax\ifhmode\unskip\space\fi MR }
\providecommand{\MRhref}[2]{%
  \href{http://www.ams.org/mathscinet-getitem?mr=#1}{#2}
}
\providecommand{\href}[2]{#2}

\bibliographystyle{amsplain}

\begin{thebibliography}{10}

\bibitem{Aavatsmark-etal-1998a}
I.~Aavatsmark, T.~Barkve, {\O}.~B{\o}e, and T.~Mannseth, \emph{Discretization
  on unstructured grids for inhomogeneous, anisotropic media. {I}. {D}erivation
  of the methods}, SIAM J. Sci. Comput. \textbf{19} (1998), no.~5, 1700--1716.
  \MR{1618761 (99f:65138)}

\bibitem{Aavatsmark-etal-1998b}
\bysame, \emph{Discretization on unstructured grids for inhomogeneous,
  anisotropic media. {II}. {D}iscussion and numerical results}, SIAM J. Sci.
  Comput. \textbf{19} (1998), no.~5, 1717--1736. \MR{1611742 (99f:65139)}

\bibitem{Aavatsmark-2002}
Ivar Aavatsmark, \emph{An introduction to multipoint flux approximations for
  quadrilateral grids}, Comput. Geosci. \textbf{6} (2002), no.~3-4, 405--432,
  Locally conservative numerical methods for flow in porous media. \MR{1956024}

\bibitem{Aavatsmark2007}
\bysame, \emph{Interpretation of a two-point flux stencil for skew
  parallelogram grids}, Computational Geosciences \textbf{11} (2007), no.~3,
  199--206.

\bibitem{Arnold-Awanou-2014}
Douglas~N. Arnold and Gerard Awanou, \emph{Finite element differential forms on
  cubical meshes}, Math. Comp. \textbf{83} (2014), no.~288, 1551--1570.
  \MR{3194121}

\bibitem{Arnold-Boffi-Bonizzoni-2015}
Douglas~N. Arnold, Daniele Boffi, and Francesca Bonizzoni, \emph{Finite element
  differential forms on curvilinear cubic meshes and their approximation
  properties}, Numer. Math. \textbf{129} (2015), no.~1, 1--20. \MR{3296150}

\bibitem{AFW06}
Douglas~N. Arnold, Richard~S. Falk, and Ragnar Winther, \emph{Finite element
  exterior calculus, homological techniques, and applications}, Acta Numer.
  \textbf{15} (2006), 1--155. \MR{2269741 (2007j:58002)}

\bibitem{AFW10}
\bysame, \emph{Finite element exterior calculus: from {H}odge theory to
  numerical stability}, Bull. Amer. Math. Soc. (N.S.) \textbf{47} (2010),
  no.~2, 281--354. \MR{2594630 (2011f:58005)}

\bibitem{Baranger-Maitre-Oudin-1996}
Jacques Baranger, Jean-Fran{\c{c}}ois Maitre, and Fabienne Oudin,
  \emph{Connection between finite volume and mixed finite element methods},
  RAIRO Mod\'el. Math. Anal. Num\'er. \textbf{30} (1996), no.~4, 445--465.
  \MR{1399499 (97h:65133)}

\bibitem{Bause-Hoffman-Knabner-2010}
Markus Bause, Joachim Hoffmann, and Peter Knabner, \emph{First-order
  convergence of multi-point flux approximation on triangular grids and
  comparison with mixed finite element methods}, Numer. Math. \textbf{116}
  (2010), no.~1, 1--29. \MR{2660444 (2011g:65228)}

\bibitem{MR3097958}
Daniele Boffi, Franco Brezzi, and Michel Fortin, \emph{Mixed finite element
  methods and applications}, Springer Series in Computational Mathematics,
  vol.~44, Springer, Heidelberg, 2013. \MR{3097958}

\bibitem{Brenner-Scott-book}
Susanne~C. Brenner and L.~Ridgway Scott, \emph{The mathematical theory of
  finite element methods}, {T}hird ed., Springer, 2008. \MR{515228 (80k:35056)}

\bibitem{BFBook}
F.~Brezzi and M.~Fortin, \emph{Mixed and hybrid finite element methods},
  Springer Series in computational Mathematics, vol.~15, Springer, 1992.
  \MR{MR2233925 (2008i:35211)}

\bibitem{Brezzi-Fortin-Maridi-2006}
F.~Brezzi, M.~Fortin, and L.~D. Marini, \emph{Error analysis of piecewise
  constant pressure approximations of {D}arcy's law}, Comput. Methods Appl.
  Mech. Engrg. \textbf{195} (2006), no.~13-16, 1547--1559. \MR{2203980
  (2006j:76146)}

\bibitem{BDDF}
Franco Brezzi, Jim Douglas, Jr., Ricardo Dur{\'a}n, and Michel Fortin,
  \emph{Mixed finite elements for second order elliptic problems in three
  variables}, Numer. Math. \textbf{51} (1987), no.~2, 237--250. \MR{890035
  (88f:65190)}

\bibitem{BDM85}
Franco Brezzi, Jr.~Jim Douglas, and L.~D. Marini, \emph{Two families of mixed
  finite elements for second order elliptic problems}, Numer. Math. \textbf{47}
  (1985), no.~2, 217--235. \MR{799685 (87g:65133)}

\bibitem{Brezzi-Lipnikov-Shashkov-2005}
Franco Brezzi, Konstantin Lipnikov, and Mikhail Shashkov, \emph{Convergence of
  the mimetic finite difference method for diffusion problems on polyhedral
  meshes}, SIAM J. Numer. Anal. \textbf{43} (2005), no.~5, 1872--1896.
  \MR{2192322}

\bibitem{Brezzi-Lipnikov-Simoncini-2005}
Franco Brezzi, Konstantin Lipnikov, and Valeria Simoncini, \emph{A family of
  mimetic finite difference methods on polygonal and polyhedral meshes}, Math.
  Models Methods Appl. Sci. \textbf{15} (2005), no.~10, 1533--1551.
  \MR{2168945}

\bibitem{Christiansen-Gillette-2016}
Snorre~H. Christiansen and Andrew Gillette, \emph{Constructions of some minimal
  finite element systems}, ESAIM Math. Model. Numer. Anal. \textbf{50} (2016),
  no.~3, 833--850. \MR{3507275}

\bibitem{Cockburn-Qiu-2014}
Bernardo Cockburn and Weifeng Qiu, \emph{Commuting diagrams for the {TNT}
  elements on cubes}, Math. Comp. \textbf{83} (2014), no.~286, 603--633.
  \MR{3143686}

\bibitem{Cohen-Joly-Roberts-Tordjman-2001}
G.~Cohen, P.~Joly, J.~E. Roberts, and N.~Tordjman, \emph{Higher order
  triangular finite elements with mass lumping for the wave equation}, SIAM J.
  Numer. Anal. \textbf{38} (2001), no.~6, 2047--2078 (electronic). \MR{1856242}

\bibitem{MonkCohen1998}
Gary Cohen and Peter Monk, \emph{Gauss point mass lumping schemes for
  {M}axwell's equations}, Numer. Methods Partial Differential Equations
  \textbf{14} (1998), no.~1, 63--88. \MR{1601785}

\bibitem{DEC05}
Mathieu Desbrun, Anil~N. Hirani, Melvin Leok, and Jerrold~E. Marsden,
  \emph{Discrete exterior calculus}, preprint (2005).

\bibitem{Droniou-2014}
Jerome Droniou, \emph{Finite volume schemes for diffusion equations:
  introduction to and review of modern methods}, Math. Models Methods Appl.
  Sci. \textbf{24} (2014), no.~8, 1575--1619. \MR{3200243}

\bibitem{Droniou-Eymard-2006}
J{\'e}r{\^o}me Droniou and Robert Eymard, \emph{A mixed finite volume scheme
  for anisotropic diffusion problems on any grid}, Numer. Math. \textbf{105}
  (2006), no.~1, 35--71. \MR{2257385}

\bibitem{Eymard-Gallouet-Herbin-2000}
Robert Eymard, Thierry Gallou{\"e}t, and Rapha{\`e}le Herbin, \emph{Finite
  volume methods}, Handbook of numerical analysis, {V}ol. {VII}, Handb. Numer.
  Anal., VII, North-Holland, Amsterdam, 2000, pp.~713--1020. \MR{1804748}

\bibitem{Gillette-Kloefkorn-2016}
Andrew Gillette and Tyler Kloefkorn, \emph{Trimmed serendipity finite element
  differential forms}, preprint (2016).

\bibitem{Hiptmair-2001}
R.~Hiptmair, \emph{Discrete {H}odge operators}, Numer. Math. \textbf{90}
  (2001), no.~2, 265--289. \MR{1872728 (2002m:39002)}

\bibitem{Hirani-thesis}
Anil~Nirmal Hirani, \emph{Discrete exterior calculus}, ProQuest LLC, Ann Arbor,
  MI, 2003, Thesis (Ph.D.)--California Institute of Technology. \MR{2704508}

\bibitem{MR2915563}
Michael Holst and Ari Stern, \emph{Geometric variational crimes: {H}ilbert
  complexes, finite element exterior calculus, and problems on hypersurfaces},
  Found. Comput. Math. \textbf{12} (2012), no.~3, 263--293. \MR{2915563}

\bibitem{Ingram-Wheeler-Yotov-2010}
Ross Ingram, Mary~F. Wheeler, and Ivan Yotov, \emph{A multipoint flux mixed
  finite element method on hexahedra}, SIAM J. Numer. Anal. \textbf{48} (2010),
  no.~4, 1281--1312. \MR{2684336 (2012a:65333)}

\bibitem{Klausen-Winther-2006b}
Runhild~A. Klausen and Ragnar Winther, \emph{Convergence of multipoint flux
  approximations on quadrilateral grids}, Numer. Methods Partial Differential
  Equations \textbf{22} (2006), no.~6, 1438--1454. \MR{2257642 (2008a:65218)}

\bibitem{Klausen-Winther-2006a}
\bysame, \emph{Robust convergence of multi point flux approximation on rough
  grids}, Numer. Math. \textbf{104} (2006), no.~3, 317--337. \MR{2244356
  (2007j:65096)}

\bibitem{MR1165446}
Jian~Ming Miao and Adi Ben-Israel, \emph{On principal angles between subspaces
  in {${\bf R}^n$}}, Linear Algebra Appl. \textbf{171} (1992), 81--98.
  \MR{1165446 (93i:15028)}

\bibitem{Wheeler-Xue-Yotov-2012}
Mary Wheeler, Guangri Xue, and Ivan Yotov, \emph{A multipoint flux mixed finite
  element method on distorted quadrilaterals and hexahedra}, Numer. Math.
  \textbf{121} (2012), no.~1, 165--204. \MR{2909918}

\bibitem{Wheeler-Yotov-2006}
Mary~F. Wheeler and Ivan Yotov, \emph{A multipoint flux mixed finite element
  method}, SIAM J. Numer. Anal. \textbf{44} (2006), no.~5, 2082--2106.
  \MR{2263041 (2008f:65228)}

\end{thebibliography}
\vspace{.125in}

\end{document}